\documentclass[pdflatex,sn-mathphys-num]{sn-jnl}

\usepackage{amsmath,amssymb,amsfonts,amsthm}
\usepackage{mathrsfs}
\usepackage{mathtools}
\usepackage{booktabs}
\usepackage{array}
\usepackage{float}
\usepackage{microtype}
\usepackage{enumitem}
\usepackage{placeins}
\usepackage[caption=false]{subfig}
\usepackage{braket}

\renewcommand{\Set}[1]{\left\{#1\right\}}
\usepackage[dvipsnames]{xcolor}
\usepackage{tikz}

\setenumerate{label=(\roman*),itemsep=3pt,topsep=3pt}

\newcolumntype{L}[1]{>{\raggedright\arraybackslash}p{#1}}

\mathtoolsset{showonlyrefs}
\allowdisplaybreaks

\newtheorem{theorem}{Theorem}[section]
\newtheorem{proposition}[theorem]{Proposition}
\newtheorem{lemma}[theorem]{Lemma}
\newtheorem{corollary}[theorem]{Corollary}
\theoremstyle{definition}

\theoremstyle{remark}

\DeclareMathOperator{\conv}{conv}
\DeclareMathOperator{\cone}{cone}
\DeclareMathOperator{\card}{card}
\DeclareMathOperator{\clconv}{\overline{conv}}

\newcommand{\R}{\mathbb R}
\newcommand{\Smat}{\mathbb S}
\newcommand{\vd}{\mathbf{d}}
\newcommand{\vx}{\mathbf{x}}
\newcommand{\vy}{\mathbf{y}}
\newcommand{\vz}{\mathbf{z}}
\newcommand{\vr}{\mathbf{r}}
\newcommand{\vnu}{\boldsymbol{\nu}}

\newcommand{\vc}{\mathbf{c}}
\newcommand{\cC}{\bar{\mathcal C}}

\DeclareRobustCommand{\Fset}{\mathcal{F}}
\DeclareRobustCommand{\Gset}{\mathcal{G}}

\DeclareRobustCommand{\Sset}{\mathcal{S}}

\newcommand{\Ncone}{\mathcal N}

\newcommand{\eps}{\varepsilon}

\newcommand{\mxRrp}{\begin{pmatrix}
		p & \vr^\top\\\vr & R
\end{pmatrix}}
\newcommand{\mxX}{\begin{pmatrix}
		1 & \vx^\top\\ \vx & X
	\end{pmatrix}}

\title[Quadratic Convexification of a Square Truncated by a Hyperbola]{Quadratic Convexification of a Square Truncated by a Hyperbola}
\author[1]{\fnm{Yipeng} \sur{Zhang}}
\author[1]{\fnm{Yuyuan} \sur{Ouyang}}
\author[1]{\fnm{Boshi} \sur{Yang}}
\affil[1]{\orgdiv{School of Mathematical and Statistical Sciences}, \orgname{Clemson University}, \orgaddress{\city{Clemson}, \postcode{29634}, \state{South Carolina}, \country{USA}}}

\begin{document}

\abstract{
We study the quadratic convexification of the compact nonconvex set
\[
\Gset
:=
\left[\tfrac12,2\right]^2
\cap
\Set{(x_1,x_2)\in\R^2\mid x_1x_2\leq1},
\]
which is a square truncated by a hyperbolic arc. Although the ordinary convex hull of \(\Gset\) is a triangle, its lifted convex hull in the complete quadratic space retains the nontrivial geometry of the curved boundary. We characterize all extreme rays of the cone of quadratic polynomials nonnegative on \(\Gset\), including parameterized families of bounded tangent and bitangent rays. Using this classification and conic duality, we derive an exact finite semidefinite representation of the lifted convex hull of \(\Gset\). The analysis follows the general framework of our earlier work on an unbounded product-constrained region, but the bounded geometry creates new boundary-contact patterns and leads to a different finite organization of the nonnegative-quadratic cone.
}

\maketitle

\section{Introduction}\label{sec:intro}

We study the compact nonconvex set
\begin{align}
	\label{eq:Gset}
	\Gset
	:=
	\left[\tfrac12,2\right]^2
	\cap
	\{(x_1,x_2) \mid x_1x_2\leq1\}
    = \{(x_1,x_2) \mid x_1 \geq \frac12, x_2 \geq \frac12, x_1x_2\leq1\}.
\end{align}
Geometrically, \(\Gset\)
is a square whose upper-right corner is removed by the hyperbolic arc
\begin{align}
	\label{eq:Gamma}
    \Gamma
    :=
    \Set{(u,u^{-1})\mid \tfrac12\leq u\leq2}.
\end{align}
Its ordinary convex hull is the triangle
\[
    \conv(\Gset)
    =
    \Set{(x_1,x_2)\in\R^2\mid
        x_1\geq\tfrac12,\ 
        x_2\geq\tfrac12,\ 
        x_1+x_2\leq\tfrac52},
\]
with vertices
\begin{align}
	\label{eq:ABC}
    A=(\tfrac12,\tfrac12),\qquad
    B=(\tfrac12,2),\qquad
    C=(2,\tfrac12).
\end{align}
Thus linear optimization cannot distinguish \(\Gset\) from a triangle,
whereas quadratic objectives can detect the curved boundary and the portion
of the triangle lying above it. This contrast makes \(\Gset\) a simple
setting in which the ordinary convex hull is elementary but the quadratic
convexification remains nontrivial.

For a set \(\Sset\subseteq\R^n\), its closed lifted convex hull in the
complete quadratic space is
\[
    \cC(\Sset)
    :=
    \clconv\left\{
        \begin{pmatrix}1\\ \vx\end{pmatrix}
        \begin{pmatrix}1\\ \vx\end{pmatrix}^{\!\top}
        \;\middle|\;
        \vx\in\Sset
    \right\}.
\]
Writing
\[
    Y=
    \begin{pmatrix}
        1&\vx^\top\\
        \vx&X
    \end{pmatrix},
\]
a quadratic function
\(\vx^\top Q\vx+2\vc^\top\vx+\gamma\) becomes the linear function
\(Q\mathbin{\bullet}X+2\vc^\top\vx+\gamma\). Consequently, an exact
description of \(\cC(\Sset)\) provides a convex reformulation for every
quadratic objective over \(\Sset\).

Our analysis approaches this lifted hull through its dual cone of
nonnegative quadratic functions. For \(R\in\Smat^n\), \(\vr\in\R^n\),
and \(p\in\R\), let
\begin{equation}\label{eq:q-def-intro}
    q_{R,\vr,p}(\vx)
    :=
    \vx^\top R\vx+2\vr^\top\vx+p,
\end{equation}
and define
\[
    \Ncone(\Sset)
    :=
    \Set{
        \begin{pmatrix}
            p&\vr^\top\\
            \vr&R
        \end{pmatrix}
        \;\middle|\;
        q_{R,\vr,p}(\vx)\geq0
        \quad\forall\,\vx\in\Sset
    }.
\]
We identify a quadratic with its coefficient matrix whenever no confusion
can arise. Every member of \(\Ncone(\Sset)\) gives a valid inequality
\[
    R\mathbin{\bullet}X+2\vr^\top\vx+p\geq0
\]
for \(\cC(\Sset)\). To state the duality precisely, define
\begin{equation}\label{eq:KF}
    \mathcal K(\Sset)
    :=
    \overline{\cone}\left\{
        \begin{pmatrix}1\\ \vx\end{pmatrix}
        \begin{pmatrix}1\\ \vx\end{pmatrix}^{\!\top}
        \;\middle|\;
        \vx\in\Sset
    \right\}.
\end{equation}
Then
\begin{equation}\label{eq:KNC_relations}
    \mathcal K(\Sset)^*=\Ncone(\Sset),
    \qquad
    \mathcal K(\Sset)=\Ncone(\Sset)^*,
    \qquad
    \cC(\Sset)
    =
    \mathcal K(\Sset)\cap\Set{Y\mid Y_{00}=1}.
\end{equation}
The extreme rays of \(\Ncone(\Sset)\) therefore identify the irreducible
sources of valid inequalities for the lifted convex hull.

Exact descriptions of complete quadratic lifts are known for several
low-dimensional structures. In particular, simplices can be
treated using positive-semidefinite constraints together with entrywise
nonnegativity or reformulation--linearization technique constraints
\cite{AnstreicherBurer2010,sherali2013reformulation}. A related literature
studies the bilinear graph \(z=x_1x_2\) over a box. Imposing an additional
upper or lower bound on \(z\) produces inequalities beyond the standard
McCormick description \cite{belotti2010valid}, and the resulting
parameterized families can be represented using second-order-cone
constraints \cite{AnstreicherBurerPark2021}. Those results concern the
three-dimensional bilinear graph. The present problem instead asks for the
complete lift
\[
    (x_1,x_2,x_1^2,x_1x_2,x_2^2),
\]
including the interaction between the product moment and both diagonal
quadratic moments.

This paper is a bounded companion to our study of
\[
    \Fset
    :=
    \Set{(x_1,x_2)\in\R_+^2\mid x_1x_2\leq1}
\]
in \cite{zhang2026nonnegative}. The methodological similarity is
intentional: in both papers we first classify the extreme rays of the cone
of nonnegative quadratics and then derive the lifted convex hull by conic
duality. The bounded result, however, is not obtained by simply restricting
the formulation for \(\Fset\). Since \(\Gset\subset\Fset\),
\(\Ncone(\Gset)\) is a larger cone, and truncation changes the relevant
contact geometry. The two unbounded coordinate axes are replaced by finite
line segments meeting the hyperbola at \(B\) and \(C\). Axis contacts and
asymptotic behavior are consequently replaced by endpoint multiplicities,
vertex contacts, and interactions among three compact boundary pieces.
Compactness also makes it possible to reduce boundary nonnegativity to
univariate polynomial nonnegativity on the fixed interval
\([\tfrac12,2]\).

A further motivation comes from the box-constrained product region
\[
    \mathcal B
    :=
    \Set{(x_1,x_2)\in[0,2]^2\mid x_1x_2\leq1}.
\]
The product constraint is redundant whenever \(x_1\leq\tfrac12\) or
\(x_2\leq\tfrac12\), and hence
\[
    \mathcal B
    =
    \bigl([0,\tfrac12]\times[0,2]\bigr)
    \cup
    \bigl([0,2]\times[0,\tfrac12]\bigr)
    \cup
    \Gset.
\]
The first two pieces are boxes, whose complete quadratic lifts are already
well understood, whereas \(\Gset\) is the only piece with a curved
boundary. Consequently, the quadratic convexification of \(\Gset\)
provides the nontrivial component in an exact disjunctive formulation of
the lifted convex hull of \(\mathcal B\). In this sense, studying
\(\Gset\) isolates the essential difficulty created by imposing a product
upper bound within a bounded box.

Our first contribution is a complete characterization of the extreme rays
of \(\Ncone(\Gset)\). Besides affine-square rays, the classification
contains products of the affine functions defining the sides of
\(\conv(\Gset)\), the hyperbola ray \(1-x_1x_2\), a one-parameter family
of bounded lifted tangent rays, and two symmetric two-parameter families of
bounded bitangent rays. The tangent rays have zero sets passing through
\(A,B,C\) and tangent to \(\Gamma\) at an additional point. The bitangent
rays combine a vertex contact with two second-order boundary contacts, one
on a flat side and one on the hyperbolic arc. These contact patterns are
specific to the bounded geometry and differ from those appearing in the
unbounded companion problem.

The proof separates quadratics according to whether their quadratic-part
matrix is positive-semidefinite (PSD). In the PSD case,
extremality reduces to the affine-square rays. In the non-PSD case when quadratic part has a
negative direction, nonnegativity on \(\Gset\) is determined by
nonnegativity on its boundary. The restrictions to the two flat sides are
univariate quadratics, while the restriction to \(\Gamma\), after clearing
denominators, is a polynomial of degree at most four. Extremality can then
be analyzed through the locations and multiplicities of their boundary
zeros. A finite contact analysis produces the bounded tangent and bitangent
families and rules out all remaining configurations.

Our second contribution is an exact finite semidefinite representation of
\(\cC(\Gset)\). Although the extreme-ray description contains continuous
families, all non-square extreme rays vanish at \(B\) or at \(C\). Let
\[
    \Ncone_B
    :=
    \Set{q\in\Ncone(\Gset)\mid q(B)=0},
    \qquad
    \Ncone_C
    :=
    \Set{q\in\Ncone(\Gset)\mid q(C)=0}.
\]
The classification gives the finite organization
\[
    \Ncone(\Gset)
    =
    \Smat_+^3+\Ncone_B+\Ncone_C.
\]
For a quadratic vanishing at \(B\), we show that nonnegativity on
\(\Gset\) is equivalent to nonnegativity on its three boundary pieces,
including when the quadratic part is positive semidefinite. The zero at
\(B\) forces factors in two boundary restrictions. After removing these
factors, membership in \(\Ncone_B\) reduces to nonnegativity of a linear,
a quadratic, and a cubic polynomial on \([\tfrac12,2]\). Fixed-size
Markov--Luk\'acs representations of these univariate conditions yield a
semidefinite representation of \(\Ncone_B\); symmetry treats
\(\Ncone_C\), and duality produces the formulation of
\(\cC(\Gset)\). This endpoint-based decomposition is different from the
direct compression of the tangent and bitangent families used in the
unbounded setting.

The resulting formulation is an exact convex reformulation for every
quadratic objective over \(\Gset\). It can also be used as a local
strengthening in a larger QCQP whenever a pair of variables, after scaling,
satisfies the constraints defining \(\Gset\). The corresponding principal
submatrix of the lifted variable must then satisfy the pulled-back
description of \(\cC(\Gset)\). More broadly, the analysis of \(\Gset\) shows that a simple ordinary convex hull may conceal a substantially richer quadratic lift.

The remainder of the paper is organized as follows.
Section~\ref{sec:roadmap} studies the extreme-ray classification of $\Ncone(\Gset)$. Section~\ref{sec:G} derives the semidefinite
representation of \(\cC(\Gset)\). Section~\ref{sec:conclusion} concludes the paper.

\section{Extreme-ray classification of the nonnegativity cone}
\label{sec:extreme-ray-classification}

This section states and proves the extreme-ray classification of
\(\Ncone(\Gset)\).  We first fix the notation used in the theorem.  The
proof then treats separately quadratic parts that are positive semidefinite
(PSD) and those that have a negative direction, verifies every listed ray,
proves that the list is exhaustive, and finally establishes exact conic
generation.

\subsection{Main theorem}

Let \(AB\) and \(AC\) denote the line segments joining the corresponding
vertices in \eqref{eq:ABC}. Their relative interiors are denoted by
\(AB^\circ\) and \(AC^\circ\).  Thus
\(\partial\Gset=AB\cup AC\cup\Gamma\).  We use the following notations to describe functions of curves that form the boundary of $\Gset$ and its convex hull:
\begin{equation}\label{eq:boundary}
	\ell_{AB}=x_1-\tfrac12,
	\qquad
	\ell_{AC}=x_2-\tfrac12,
	\qquad
	\ell_{\Gamma}=1-x_1x_2,
	\qquad
	\ell_{BC}=\tfrac52-x_1-x_2.
\end{equation}
The first three describe the boundary of \(\Gset\), whereas
\(\ell_{BC}=0\) is the third side of \(\conv(\Gset)\).  In particular,
\begin{equation}\label{eq:affine-cone}
	\Set{\ell\mid \ell\text{ is affine and }\ell\geq0\text{ on }\Gset}
	=
	\cone\{\ell_{AB},\ell_{AC},\ell_{BC}\}.
\end{equation}
Indeed, an affine function is nonnegative on \(\Gset\) if and only if it
is nonnegative on the triangle \(\conv(\Gset)\).  The coefficients of the conic representation depends on the value of $\ell$ at vertices $A$, $B$, and $C$. Specifically, the three coefficients of conic representation
in \eqref{eq:affine-cone} are
\(\tfrac23\ell(C)\), \(\tfrac23\ell(B)\), and
\(\tfrac23\ell(A)\), respectively.

We can now state the theorem on classification of extreme rays below.
\begin{theorem}[Extreme rays of $\Ncone(\Gset)$]
	\label{thm:main}
	The extreme rays of \(\Ncone(\Gset)\) are precisely the rays spanned by
	the following quadratics:
	\begin{enumerate}
		\item The extreme affine squares \(\ell^2\), where \(\ell\) is a nonzero affine function whose zero set contains two
		distinct points of \(\Gset\). 
		\item The convex hull boundary-product rays
		\[
			\ell_{AB}\ell_{BC},
			\qquad
			\ell_{AC}\ell_{BC},
			\qquad
			\ell_{AB}\ell_{AC}.
		\]
		\item The hyperbola ray \(\ell_{\Gamma}\).
		\item The bounded lifted tangent rays
		\begin{equation}\label{eq:f}
			f_u
			:=
			\ell_{AB}(2-x_1)
			+u^2\ell_{AC}(2-x_2)
			-4u\ell_{AB}\ell_{AC},
			\qquad
			\tfrac12\leq u\leq2.
		\end{equation}
		\item The bounded bitangent rays
		\begin{equation}\label{eq:g}
			\begin{aligned}
			g^1_{v,w}
			&:=
			\alpha_{v,w}
			\left(x_1+\frac v2x_2-v-\frac12\right)^2
			+(\alpha_{v,w}-1)f_v
			+\beta_{v,w}\ell_{\Gamma},
			\\
			& \text{where }\alpha_{v,w}
			:=\frac{16(2-w)^2}{9(2-v)^2},
			\qquad
			\beta_{v,w}:=4w-5+3(1-v)\alpha_{v,w},
			\\
			& \qquad\quad\frac12\leq v<2,
			\qquad
			\frac12\leq w<\frac{2+3v}{4},
			\end{aligned}
		\end{equation}
		and $g^2_{v,w}(x_1,x_2):=g^1_{v,w}(x_2,x_1)$.
	\end{enumerate}
	Moreover, \(\Ncone(\Gset)\) is the conic hull of these rays.
\end{theorem}

\begin{figure}[!htbp]
	\centering
	\begingroup
	\newcommand{\Gbase}{%
		\coordinate (A) at (0.5,0.5);
		\coordinate (B) at (0.5,2);
		\coordinate (C) at (2,0.5);
		\path[fill=blue!12]
		(A) -- (B)
		-- plot[domain=0.5:2,samples=120,variable=\x]
		({\x},{1/\x})
		-- (C) -- cycle;
		\draw[->,gray!70] (-0.5,0) -- (3,0)
		node[right] {\scriptsize \(x_1\)};
		\draw[->,gray!70] (0,-0.5) -- (0,3)
		node[above] {\scriptsize \(x_2\)};
		\draw[boundary] (A) -- (B);
		\draw[boundary] (A) -- (C);
		\draw[boundary,domain=0.5:2,samples=120,smooth,variable=\x]
		plot ({\x},{1/\x});
		\foreach \P in {A,B,C}
		\node[vertex] at (\P) {};
	}
	
	\begin{minipage}[t]{0.48\textwidth}
		\centering
		\begin{tikzpicture}[
			x=1.45cm,y=1.45cm,
			boundary/.style={thick,blue!65!black},
			vertex/.style={circle,fill=black,inner sep=1.15pt},
			rayzero/.style={thick,red!75!black},
			contact/.style={circle,fill=red!75!black,inner sep=1.35pt}
			]
			\Gbase
			\node[below left]  at (A) {\scriptsize \(A\)};
			\node[above left]  at (B) {\scriptsize \(B\)};
			\node[below right] at (C) {\scriptsize \(C\)};
			\node at (1.15,1.15)
			{\scriptsize \(\Gamma\)};
			\node[blue!65!black] at (1.0,0.7)
			{\scriptsize \(\Gset\)};
		\end{tikzpicture}
		
	\end{minipage}\hfill
	\begin{minipage}[t]{0.48\textwidth}
		\centering
		\begin{tikzpicture}[
			x=1.45cm,y=1.45cm,
			boundary/.style={thick,blue!65!black},
			vertex/.style={circle,fill=black,inner sep=1.15pt},
			rayzero/.style={thick,red!75!black},
			contact/.style={circle,fill=red!75!black,inner sep=1.35pt}
			]
			\Gbase
			\begin{scope}
				\clip (-0.5,-0.5) rectangle (3,3);
				\draw[rayzero,domain=-0.5:3,samples=280,smooth,variable=\t]
				plot ({\t},{(1/28)*(51-32*\t+sqrt(768*\t*\t-1728*\t+1113))});
				\draw[rayzero,domain=-0.5:3,samples=280,smooth,variable=\t]
				plot ({\t},{(1/28)*(51-32*\t-sqrt(768*\t*\t-1728*\t+1113))});
			\end{scope}
			\node[contact] at (A) {};
			\node[contact] at (B) {};
			\node[contact] at (C) {};
			\node[contact] at (1.75,{4/7}) {};
			\node[red!75!black] at (2,0.95)
			{\scriptsize \(f_{7/4}(\vx)=0\)};
			\node[below left]  at (A) {\scriptsize \(A\)};
			\node[above left]  at (B) {\scriptsize \(B\)};
			\node[below right] at (C) {\scriptsize \(C\)};
		\end{tikzpicture}
		
	\end{minipage}
	
	\par\medskip
	
	\begin{minipage}[t]{0.48\textwidth}
		\centering
		\begin{tikzpicture}[
			x=1.45cm,y=1.45cm,
			boundary/.style={thick,blue!65!black},
			vertex/.style={circle,fill=black,inner sep=1.15pt},
			rayzero/.style={thick,red!75!black},
			contact/.style={circle,fill=red!75!black,inner sep=1.35pt}
			]
			\Gbase
			\begin{scope}
				\clip (-0.5,-0.5) rectangle (3,3);
				\draw[rayzero,domain=-0.5:3,samples=280,smooth,variable=\t]
				plot ({\t},{(6/13)*((-7*\t)/3+23/6-sqrt((88*\t*\t-174*\t+90)/9))});
				\draw[rayzero,domain=-0.5:3,samples=280,smooth,variable=\t]
				plot ({\t},{(6/13)*((-7*\t)/3+23/6+sqrt((88*\t*\t-174*\t+90)/9))});
			\end{scope}
			\node[contact] at (B) {};
			\node[contact] at (0.75,0.5) {};
			\node[contact] at (1,1) {};
			\node[red!75!black] at (1.55,1.35)
			{\scriptsize \(g^{1}_{1,\,3/4}(\vx)=0\)};
			\node[below left]  at (A) {\scriptsize \(A\)};
			\node[above left]  at (B) {\scriptsize \(B\)};
			\node[below right] at (C) {\scriptsize \(C\)};
		\end{tikzpicture}
		
	\end{minipage}\hfill
	\begin{minipage}[t]{0.48\textwidth}
		\centering
		\begin{tikzpicture}[
			x=1.45cm,y=1.45cm,
			boundary/.style={thick,blue!65!black},
			vertex/.style={circle,fill=black,inner sep=1.15pt},
			rayzero/.style={thick,red!75!black},
			contact/.style={circle,fill=red!75!black,inner sep=1.35pt}
			]
			\Gbase
			\begin{scope}
				\clip (-0.5,-0.5) rectangle (3,3);
				\draw[rayzero,domain=-0.5:3,samples=280,smooth,variable=\t]
				plot ({(6/13)*((-7*\t)/3+23/6-sqrt((88*\t*\t-174*\t+90)/9))},{\t});
				\draw[rayzero,domain=-0.5:3,samples=280,smooth,variable=\t]
				plot ({(6/13)*((-7*\t)/3+23/6+sqrt((88*\t*\t-174*\t+90)/9))},{\t});
			\end{scope}
			\node[contact] at (C) {};
			\node[contact] at (0.5,0.75) {};
			\node[contact] at (1,1) {};
			\node[red!75!black] at (1.75,1.25)
			{\scriptsize \(g^{2}_{1,\,3/4}(\vx)=0\)};
			\node[below left]  at (A) {\scriptsize \(A\)};
			\node[above left]  at (B) {\scriptsize \(B\)};
			\node[below right] at (C) {\scriptsize \(C\)};
		\end{tikzpicture}
		
	\end{minipage}

	\par\medskip

	\begin{minipage}[t]{0.48\textwidth}
		\centering
		\begin{tikzpicture}[
			x=1.45cm,y=1.45cm,
			boundary/.style={thick,blue!65!black},
			vertex/.style={circle,fill=black,inner sep=1.15pt},
			rayzero/.style={thick,red!75!black},
			contact/.style={circle,fill=red!75!black,inner sep=1.35pt}
			]
			\Gbase
			\begin{scope}
				\clip (-0.5,-0.5) rectangle (3,3);
				\draw[rayzero,domain=-0.5:3,samples=280,smooth,variable=\t]
				plot ({\t},{-2*\t+9/4-(1/4)*sqrt(256*\t*\t-304*\t+97)});
				\draw[rayzero,domain=-0.5:3,samples=280,smooth,variable=\t]
				plot ({\t},{-2*\t+9/4+(1/4)*sqrt(256*\t*\t-304*\t+97)});
			\end{scope}
			\node[contact] at (A) {};
			\node[contact] at (B) {};
			\node[below left] at (A) {\scriptsize \(A\)};
			\node[above left] at (B) {\scriptsize \(B\)};
			\node[below right] at (C) {\scriptsize \(C\)};
			\node[red!75!black,align=center] at (1.45,2.02)
			{\scriptsize \(g^1_{1/2,\,1/2}(\vx)=0\)};
		\end{tikzpicture}

	\end{minipage}
	\hfill
	\begin{minipage}[t]{0.48\textwidth}
		\centering
		\begin{tikzpicture}[
			x=1.45cm,y=1.45cm,
			boundary/.style={thick,blue!65!black},
			vertex/.style={circle,fill=black,inner sep=1.15pt},
			rayzero/.style={thick,red!75!black},
			contact/.style={circle,fill=red!75!black,inner sep=1.35pt}
			]
			\Gbase
			\begin{scope}
				\clip (-0.5,-0.5) rectangle (3,3);
				\pgfmathsetmacro{\xturnleft}{(85-3*sqrt(33))/128}
				\pgfmathsetmacro{\xturnright}{(85+3*sqrt(33))/128}
				\draw[rayzero]
				plot[domain=-0.5:\xturnleft,samples=150,smooth,variable=\t]
				({\t},{-58*\t/11+207/44-sqrt(max(0,1024*\t*\t-1360*\t+433))/44})
				plot[domain=\xturnleft:-0.5,samples=150,smooth,variable=\t]
				({\t},{-58*\t/11+207/44+sqrt(max(0,1024*\t*\t-1360*\t+433))/44});
				\draw[rayzero]
				plot[domain=3:\xturnright,samples=200,smooth,variable=\t]
				({\t},{-58*\t/11+207/44-sqrt(max(0,1024*\t*\t-1360*\t+433))/44})
				plot[domain=\xturnright:3,samples=200,smooth,variable=\t]
				({\t},{-58*\t/11+207/44+sqrt(max(0,1024*\t*\t-1360*\t+433))/44});
			\end{scope}
			\node[contact] at (B) {};
			\node[contact] at (0.8,0.5) {};
			\node[below left] at (A) {\scriptsize \(A\)};
			\node[above left] at (B) {\scriptsize \(B\)};
			\node[below right] at (C) {\scriptsize \(C\)};
			\node[red!75!black,align=center] at (1.35,1.82)
			{\scriptsize \(g^1_{1/2,\,4/5}(\vx)=0\)};
		\end{tikzpicture}

	\end{minipage}
	\endgroup
	\caption{\label{fig:rays}
		The bounded region \(\Gset\) and the zero set of some
		extreme rays.  Top left: the region \(\Gset\) with its three
		vertices \(A,B,C\) and the hyperbolic arc \(\Gamma\).
		Top right: the zero set of the bounded lifted tangent ray
		\(f_{7/4}\).  Middle left: the zero
		set of the bounded bitangent ray \(g^{1}_{1,\,3/4}\). Middle right: the zero set of the bounded bitangent 
		ray \(g^{2}_{1,\,3/4}\).
		Bottom left: the zero set of the bounded bitangent ray \(g^{1}_{1/2,\,1/2}\). Bottom right: the zero set of the bounded bitangent 
		ray \(g^{1}_{1/2,\,4/5}\).
	}
\end{figure}
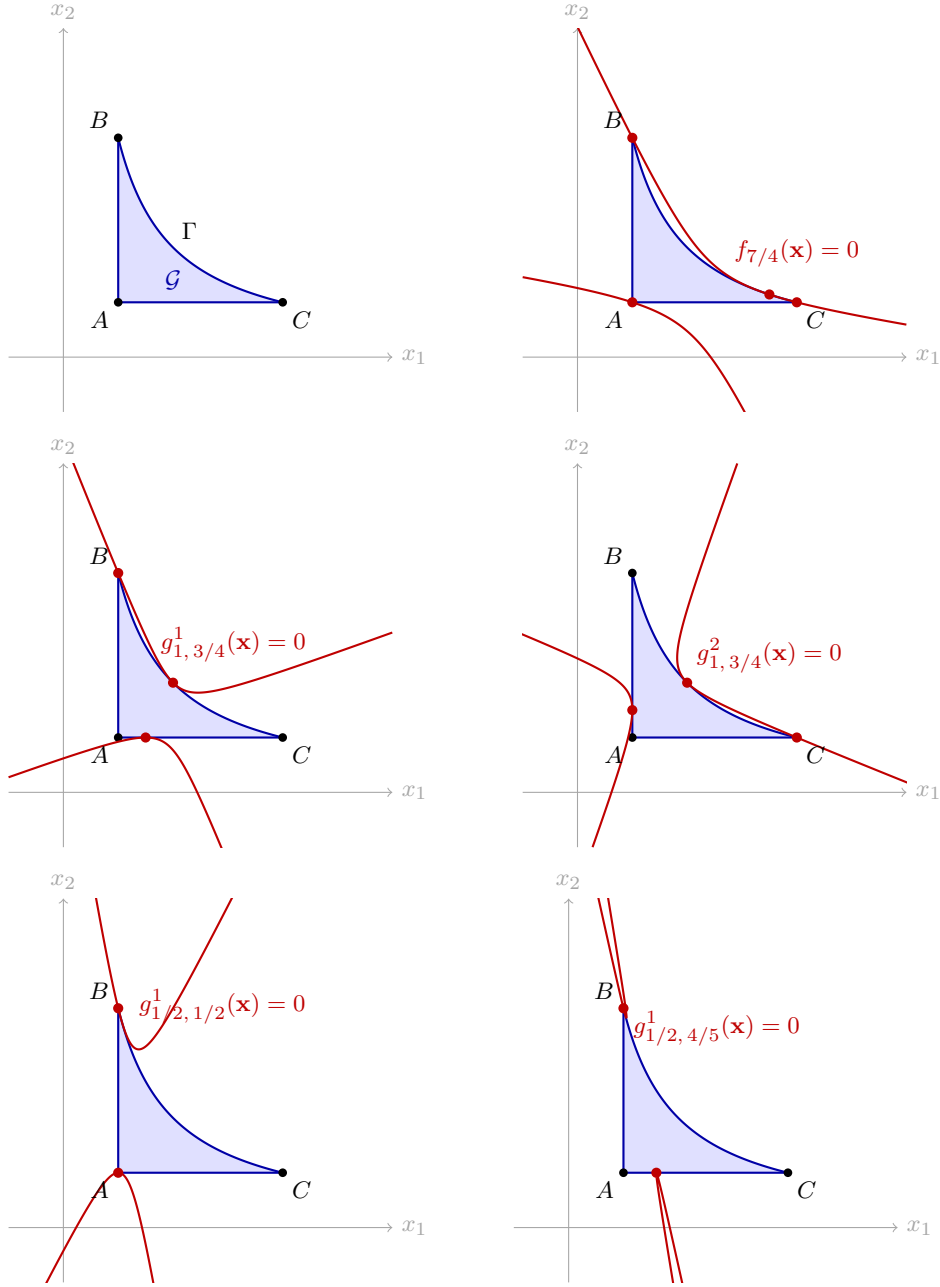

A few remarks are in place.
First, 
The zero sets of bounded lifted tangent rays $f_u$ always pass through all vertices $A$, $B$, and $C$, and is tangent to a relative interior point $(u,1/u)$ of the hyperbolic arc $\Gamma$. 
The term \emph{bounded lifted tangent ray} adapts the terminology
\emph{lifted tangent} from \cite{belotti2010valid} (see also \cite{AnstreicherBurerPark2021} and \cite{zhang2026nonnegative}). 
In their setting, the zero set of a
lifted tangent ray is tangent to the hyperbola \(x_1x_2=1\) at one point.
Here, the zero set meets that hyperbola in the nonnegative quadrant at \(B\), \(C\), and one relative interior point of the hyperbolic arc $\Gamma$. The first two intersections are
transverse and the third is tangent. Thus the word \emph{bounded}
records that the tangency lies on the bounded hyperbolic side \(\Gamma\).
The top-right panel of Figure~\ref{fig:rays} illustrates
\(f_{7/4}\).
Second,
the zero set of a bounded bitangent ray contains a distinguished vertex (either $B$ or $C$) and is tangential to two contact points on the boundary, one of which lies on the hyperbolic side
\(\Gamma\).  The tangential contact points
may coincide with vertices.
The middle panels of Figure~\ref{fig:rays} give examples of $g_{1,3/4}^1$ and $g_{1,3/4}^2$.
The term \emph{bounded bitangent ray} adapts the terminology
\emph{bitangent} from our previous work \cite{zhang2026nonnegative}. Note that one tangential contact point on the hyperbolic side $\Gamma$ might coincide with a vertex (either $B$ or $C$) that the zero set contains when the parameter $v=1/2$; see the bottom panels of Figure \ref{fig:rays} for examples of $g_{1/2,1/2}^1$ and $g_{1/2,4/5}^1$. 
Third,
all extreme rays are genuinely quadratic; no nonzero affine or constant
function spans an extreme ray. For example,
\begin{equation}\label{eq:affine-decomposition}
	\ell_{AB}
	=
	\frac23\left(
	\ell_{AB}^2
	+\ell_{AB}\ell_{AC}
	+\ell_{AB}\ell_{BC}
	\right)
\end{equation}
is a conic combination of affine square and convex hull boundary-product rays. 
Fourth,
in the definition of bounded bitangent rays we exclude some end points in parameters $v$ and $w$, since they represent rays that already appeared in the list. For example, for any $v\in [1/2,2)$, when $w=(2+3v)/4$ we have
\begin{equation}\label{eq:g-overlap}
	g^1_{v,{(2+3v)}/{4}}(\vx) 
	=
	\left(x_1+\frac v2x_2-v-\frac12\right)^2.
\end{equation}
This is an extreme affine squares ray whose associated affine function passes through $(v+1/2, 0)$ and $(v,1/v)$.

\subsection{Road map of the proof}
\label{sec:roadmap}

The proof of Theorem~\ref{thm:main} is organized according to the
quadratic-part matrix \(R\) and the zeros of extreme rays. We first verify that all the listed quadratics belong to \(\Ncone(\Gset)\) and span extreme rays. We then classify the extreme rays with \(R\succeq0\). The main part of the argument treats the remaining case
\(R\not\succeq0\), where extremality can be analyzed through the locations and multiplicities of boundary zeros. These steps establish that the list in Theorem~\ref{thm:main} is exhaustive. A compact-base argument then shows that the listed rays generate \(\Ncone(\Gset)\).

The separation of the PSD and non-PSD cases is related to the locations of the zeros of extreme rays. If a quadratic \(q_{R,\vr,p}\in\Ncone(\Gset)\) has a zero in \(\operatorname{int}(\Gset)\), then that zero is an unconstrained local minimizer, and the second-order necessary condition gives \(R\succeq0\). Thus, when \(R\not\succeq0\), every zero of \(q_{R,\vr,p}\) lies on \(\partial\Gset\). In this case, nonnegativity on \(\Gset\) is
completely determined by nonnegativity on its boundary, as stated in
Lemma~\ref{lem:boundary-reduction-bounded}. The non-PSD analysis can therefore be reduced to three one-variable polynomials associated with the boundary pieces \(AB\), \(AC\), and \(\Gamma\).

For \(1/2\leq t\leq2\), define the three boundary restrictions
\begin{equation}\label{eq:Pq}
	\begin{aligned}
		b_{q\mid AB}(t)
		&:=q(\tfrac12,t)
		=R_{22}t^2+(R_{12}+2r_2)t+\tfrac14R_{11}+r_1+p,
		\\
		b_{q\mid AC}(t)
		&:=q(t,\tfrac12)
		=R_{11}t^2+(R_{12}+2r_1)t+\tfrac14R_{22}+r_2+p,
		\\
		P_q(t)
		&:=t^2q(t,t^{-1})
		=R_{11}t^4+2r_1t^3+(2R_{12}+p)t^2+2r_2t+R_{22}.
	\end{aligned}
\end{equation}
The boundary restrictions above includes two quadratic $b_{q\mid AB}$, $b_{q\mid AC}$, and one quartic polynomial $P_q$. Here, the polynomial \(P_q\) is a denominator-cleared expression for the restriction to $\Gamma$. Multiplication by \(t^2\) does not change its zeros or its sign.

\begin{lemma}[Boundary property lemma]\label{lem:boundary-reduction-bounded}
	Consider any $q:=q_{R,\vr,p}$ with $R\not\succeq 0$.  Then the followings are equivalent:
	\begin{enumerate}
		\item\label{itm:q_in_NG} \(q\) is nonnegative
		on \(\Gset\);
		\item \label{itm:q_in_NpG} \(q\) is nonnegative on
		\(AB\cup AC\cup\Gamma\);
		\item \label{itm:q_in_NpG_ineq} For any $t\in [1/2,2]$, $b_{q\mid AB}(t)\ge 0$, $b_{q\mid AC}(t)\ge 0$, and $P_q(t)\ge 0$.
	\end{enumerate}
\end{lemma}

\begin{proof}
	Only the direction \ref{itm:q_in_NpG}$\implies$\ref{itm:q_in_NG} requires proof. Suppose that $q$ is nonnegative on the boundary $\partial \Gset=AB\cup AC\cup\Gamma$. Since $R\not\succeq 0$, we may choose \(d\neq0\) with
	\(d^\top Rd<0\). Fix any interior point \(\vz\in\Gset\).  The connected component of 
	\[
	\Set{s\in\R\mid \vz+sd\in\Gset}
	\]
	containing $s=0$ 
	is a compact interval, and its two endpoints map to
	\(\partial\Gset\).  The function \(s\mapsto q(\vz+sd)\) is strictly
	concave with nonnegative values at the end points. Thus its value at $s=0$ is also nonnegative, yielding \(q(\vz)\geq0\). We can now conclude that $q$ is nonnegative on $\Gset$. 
\end{proof}

Note that the above boundary property lemma does not hold in the PSD case; an example is 
$q(\vx):=
\left(x_1-3/4\right)^2
+
\left(x_2-3/4\right)^2
-1/{64}.
$
Direct computation could verify that it is nonnegative on $\partial \Gset$ but is negative in an interior point $(3/4,3/4)$. 
The geometry of $\Gset$ does allow a stronger boundary property lemma that applies to both PSD and non-PSD cases with one additional vanishing condition. Specifically, If $q$ vanishes at $B$, then
\begin{equation}
	\label{eq:vanishB}
	q(B)
	=\tfrac14R_{11}+2R_{12}+4R_{22}+r_1+4r_2+p
	=0.
\end{equation}
In particular,
$b_{q\mid AB}(2)=0$ and $P_q(\tfrac12)=0$.  Consequently,
\begin{equation}\label{eq:B-forced-factors}
	\begin{aligned}
		b_{q\mid AB}(t)
		&=(2-t)\widehat b_{q\mid AB}(t), \text{ where }
		&
		\widehat b_{q\mid AB}(t)
		&:=-R_{22}t-R_{12}-2R_{22}-2r_2,\text{ and }
		\\
		P_q(t)
		&=(t-\tfrac12)\widehat P_q(t),\text{ where }
		&
		\widehat P_q(t)
		&:=R_{11}t^3+(2r_1+\tfrac12R_{11})t^2\\
		&&&\quad +(2R_{12}+p+r_1+\tfrac14R_{11})t\\
		&&&\quad +2r_2+R_{12}+\tfrac12p
		+\tfrac12r_1+\tfrac18R_{11}.
	\end{aligned}
\end{equation}
In this paper we call $\widehat b_{q\mid AB}(t)$ and $\widehat P_q(t)$ \emph{reduced boundary restrictions} of $q$.  
Both of them are defined on $[1/2,2]$. In the following lemma, we show that the special structure of $\Gset$ also allows a stronger boundary property lemma. 

\begin{lemma}\label{lem:strong-bound-property-B}
	Consider any $q:=q_{R,\vr,p}$ that vanishes at the vertex $B$, i.e., $q(1/2,2)=0$. The following statements are equivalent:
	\begin{enumerate}
		\item\label{itm:q_in_NB} \(q\) is nonnegative
		on \(\Gset\);
		\item\label{itm:q_in_Npf_B} \(q\) is nonnegative on
		\(AB\cup AC\cup\Gamma\);
		\item\label{itm:q_in_NB_analytical} For any $t\in [1/2,2]$, $
		\widehat b_{q\mid AB}(t)\geq0,$
		 $
		b_{q\mid AC}(t)\geq0,$ and 
		$\widehat P_q(t)\geq0$.
	\end{enumerate}
	\end{lemma}

	\begin{proof}
	Here \ref{itm:q_in_Npf_B} and
	\ref{itm:q_in_NB_analytical} are equivalent by
	\eqref{eq:B-forced-factors}, and
	\ref{itm:q_in_NB} immediately implies
	\ref{itm:q_in_Npf_B}.  It remains to prove the reverse direction that \ref{itm:q_in_Npf_B} implies \ref{itm:q_in_NB}. Our proof strategy is illustrated in Figure \ref{fig:stronger_boundary_prop}.
	
	Suppose that \(q\) is nonnegative on \(\partial\Gset\) and satisfies
	\(q(B)=0\).  If \(R\not\succeq0\), the conclusion follows from
	Lemma~\ref{lem:boundary-reduction-bounded}.  We therefore assume that
	\(R\succeq0\), so that \(q\) is convex. 
	
	Write
	\[
	\nabla q(B)=(\gamma_1,\gamma_2)^\top.
	\]
	Since \(b_{q\mid AB}\) is nonnegative on
	\([\tfrac12,2]\) and vanishes at its right endpoint \(t=2\), we have
	\[
	\gamma_2=b_{q\mid AB}'(2)\leq0.
	\]
	Similarly, \(P_q\) is nonnegative on \([\tfrac12,2]\) and satisfies
	\(P_q(\tfrac12)=0\).  Hence
	\[
	P_q'(\tfrac12)\geq0.
	\]
	Using \(q(B)=0\) in the derivative of
	\(P_q(t)=t^2q(t,t^{-1})\), we obtain
	\[
	P_q'(\tfrac12)
	=\frac14\gamma_1-\gamma_2.
	\]
	Consequently,
	\begin{equation}\label{eq:NB-gradient-signs}
		\gamma_2\leq0,
		\qquad
		\gamma_1-4\gamma_2\geq0.
	\end{equation}
	
	Suppose, toward a contradiction, that \(q(\vx)<0\) for some
	\(\vx=(x_1,x_2)\in\Gset\).  Consider the segment from \(B\) to \(\vx\),
	parameterized by
	\[
	\vz(s):=(1-s)B+s\vx,
	\qquad 0\leq s\leq1.
	\]
	Convexity and \(q(B)=0\) give
	\begin{equation}\label{eq:NB-negative-segment}
		q(\vz(s))
		\leq (1-s)q(B)+s q(\vx)
		=s q(\vx)<0
		\qquad
		\forall\,s\in(0,1].
	\end{equation}
	
	Now define
	\[
	\rho(s):=z_1(s)z_2(s).
	\]
	We have
	\[
	\rho(0)=1,
	\qquad
	\rho(1)=x_1x_2\leq1,
	\]
	and direct differentiation gives
	\[
	\rho'(0)
	=
	2\left(x_1-\tfrac12\right)
	+\tfrac12(x_2-2)
	=
	\frac12(4x_1+x_2-4).
	\]
	If \(\rho'(0)>0\), then \(\rho(s)>1\) for all sufficiently small
	\(s>0\).  Since \(\rho(1)\leq1\), continuity would give some
	\(\theta\in(0,1]\) such that
	\[
	\rho(\theta)=1.
	\]
	The coordinates of \(\vz(\theta)\) lie in
	\([\tfrac12,2]\), so \(\vz(\theta)\in\Gamma\).  But
	\eqref{eq:NB-negative-segment} would then imply
	\(q(\vz(\theta))<0\), contradicting nonnegativity of \(q\) on
	\(\Gamma\).  Therefore,
	\begin{equation}\label{eq:NB-product-direction}
		4x_1+x_2\leq4.
	\end{equation}
	
	Set
	\[
	\lambda:=\gamma_1-4\gamma_2,
	\qquad
	\mu:=-\gamma_2.
	\]
	By \eqref{eq:NB-gradient-signs}, both \(\lambda\) and \(\mu\) are
	nonnegative.  A direct calculation gives
	\begin{align*}
		\left\langle\nabla q(B),\vx-B\right\rangle
		&=
		\lambda\left(x_1-\tfrac12\right)
		+\mu(4-4x_1-x_2).
	\end{align*}
	Both terms on the right are nonnegative because
	\(x_1\geq\tfrac12\) and \eqref{eq:NB-product-direction} holds.
	The first-order inequality for the convex function \(q\) therefore yields
	\[
	q(\vx)
	\geq q(B)
	+\left\langle\nabla q(B),\vx-B\right\rangle
	\geq0,
	\]
	contradicting \(q(\vx)<0\).
	
	Thus \(q\) is nonnegative on all of \(\Gset\), and hence
	\(q\in\Ncone_B\).
	
	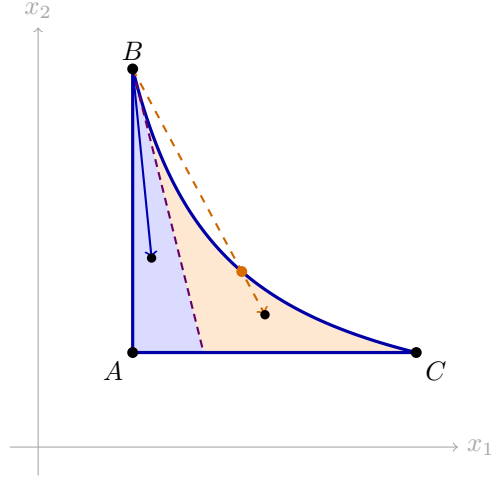
\begin{figure}
		\centering
		\begin{tikzpicture}[
			x=2.5cm,
			y=2.5cm,
			boundary/.style={very thick,blue!65!black},
			divider/.style={thick,densely dashed,violet!80!black},
			vertex/.style={circle,fill=black,inner sep=1.45pt},
			marked/.style={circle,fill=black,inner sep=1.25pt}
			]
			\coordinate (A) at (0.5,0.5);
			\coordinate (B) at (0.5,2);
			\coordinate (C) at (2,0.5);
			
			\coordinate (S) at ({7/8},{1/2});
			\coordinate (Xgrad)  at ({3/5},1);
			\coordinate (Xcross) at ({6/5},{7/10});
			\coordinate (Z)      at ({14/13},{13/14});
			
			\path[fill=blue!14]
			(A) -- (B) -- (S) -- cycle;
			
			\path[fill=orange!18]
			(S) -- (B)
			-- plot[
			domain=0.5:2,
			samples=200,
			variable=\t
			] ({\t},{1/\t})
			-- (C) -- cycle;
			
			\draw[->,gray!70]
			(-0.15,0) -- (2.22,0) node[right] {$x_1$};
			\draw[->,gray!70]
			(0,-0.15) -- (0,2.22) node[above] {$x_2$};
			
			\draw[boundary] (A) -- (B);
			\draw[boundary] (A) -- (C);
			\draw[
			boundary,
			domain=0.5:2,
			samples=200,
			smooth,
			variable=\t
			] plot ({\t},{1/\t});
			
			\draw[divider] (B) -- (S);
			
			\draw[->,thick,blue!70!black] (B) -- (Xgrad);
			\node[marked] at (Xgrad) {};
			
			\draw[->,thick,dashed,orange!80!black]
			(B) -- (Xcross);
			\node[marked] at (Xcross) {};
			
			\node[
			circle,
			fill=orange!85!black,
			inner sep=1.45pt
			] at (Z) {};
			
			\foreach \P in {A,B,C}
			\node[vertex] at (\P) {};
			
			\node[below left]  at (A) {$A$};
			\node[above]       at (B) {$B$};
			\node[below right] at (C) {$C$};
		\end{tikzpicture}
		\caption{The two regions used in the proof of
			Lemma~\ref{lem:strong-bound-property-B}.}
		\label{fig:stronger_boundary_prop}
	\end{figure}

\end{proof}

The two boundary properties above serve different purposes. 
Lemma~\ref{lem:boundary-reduction-bounded} applies to every quadratic
whose quadratic-part matrix has a negative direction, and it will be the
main boundary-reduction tool in the non-PSD extreme-ray classification.
Lemma~\ref{lem:strong-bound-property-B} imposes the additional condition
\(q(B)=0\), but in return it applies without any assumption on the
quadratic-part matrix. We use this stronger endpoint property below to
verify the nonnegativity of the bounded tangent and bitangent families.
More importantly, it will be used in Section~\ref{sec:G} to characterize
the endpoint-vanishing cone \(\Ncone_B\), which is the key step in deriving
the finite semidefinite representation of \(\cC(\Gset)\).

We next record some basic properties needed for the extreme-ray analysis.
\begin{proposition}\label{prop:has-zero}
	Every extreme ray of \(\Ncone(\Gset)\) has a zero on \(\Gset\).
\end{proposition}
\begin{proof}
	If \(q>0\) on the compact set \(\Gset\), let
	\(m:=\min_{\vx\in\Gset}q(\vx)>0\) and \(M:=\max_{\vx\in\Gset}\ell_{AB}^2(\vx)>0\).  For
	any \(0<\eps<m/M\), both \(\eps\ell_{AB}^2\) and
	\(q-\eps\ell_{AB}^2\) are nonnegative on \(\Gset\).  The affine square \(\eps\ell_{AB}^2\) has
	zeros whereas \(q\) does not, so this decomposition is nontrivial. Hence $q$ is not an extreme ray.
\end{proof}

\begin{lemma}\label{lem:contact-uniqueness}
	Let \(h\) be a quadratic that vanishes identically on \(AC\).
	\begin{enumerate}
		\item If there exists \(v \in (\frac{1}{2}, 2)\), such that
		\(h(t,t^{-1})\) has a double zero at \(t=v\) and \(h(B)=0\), then \(h=0\) .
		\item If \(P_h\) has a zero of multiplicity at least three at
		\(t=\tfrac12\), then \(h=0\).
	\end{enumerate}
\end{lemma}

\begin{proof}
	Write \(h=\ell_{AC}a\) for an affine function \(a\).  In the first case,
	\(a(B)=0\), while \(a\) vanishes at \((v,v^{-1})\) and its derivative
	along \(\Gamma\) also vanishes there.  If \(a\neq0\), its zero line must
	therefore be the tangent line
	\[
	x_1+v^2x_2-2v=0.
	\]
	Its value at \(B\) is \(2(v-\tfrac12)^2>0\), a contradiction.
	
	For the second case, \(\ell_{AC}\) does not vanish at \(B\).  Hence
	\(a(t,t^{-1})\) has a zero of multiplicity at least three at
	\(t=\tfrac12\).  Since \(t\,a(t,t^{-1})\) has degree at most two, it is
	the zero polynomial.  Thus \(a=0\), and hence \(h=0\).
\end{proof}

\begin{proposition}\label{prop:candidates}
	Every ray listed in Theorem~\ref{thm:main} belongs to
	\(\Ncone(\Gset)\) and is extreme.
\end{proposition}
\begin{proof}
	We start by showing that all rays listed in Theorem \ref{thm:main} are nonnegative over $\Gset$. The nonnegativity of extreme affine squares, convex hull boundary-product, and hyperbola rays are straightforward. 	
	For \(f_u\), direct computation yields
$\widehat b_{f_u\mid AB}(t)=u^2(t-\tfrac12)$,
		$b_{f_u\mid AC}(t)=(t-\tfrac12)(2-t)$, and $\widehat P_{f_u}(t)=(2-t)(t-u)^2$. 
	These three (reduced) restrictions are nonnegative on
	\([1/2,2]\), and $f_u$ vanishes at $B$.
	Lemma~\ref{lem:strong-bound-property-B} therefore guarantees that
	\(f_u\in\Ncone(\Gset)\). For \(g^1_{v,w}\), direct computation shows that it vanishes at the vertex $B$ and that two of the (reduced) boundary restrictions are
	\begin{equation}\label{eq:bounded:g-bottom}
		\begin{aligned}
			b_{g^1_{v,w}\mid AC}(t)=(t-w)^2\text{ and }
\widehat P_{g^1_{v,w}}(t)
			=(t-v)^2
			\left(t+\frac32\alpha_{v,w}-2\right).
		\end{aligned}
	\end{equation}
	For the remaining reduced boundary restriction $\widehat b_{g^1_{v,w}\mid AB}$, recall from \eqref{eq:B-forced-factors} that it is affine. Direct computation shows that its endpoint values at $[1/2,2]$ are
	\begin{align*}
		\widehat b_{g^1_{v,w}\mid AB}(\tfrac12)
		&= \frac{(2w-1)^2}{6},\\
		\widehat b_{g^1_{v,w}\mid AB}(2)
		&=\frac32\left(\sqrt{\alpha_{v,w}}-1\right)
		\left(
		v(2v-1)
		+(v^2-v+1)\left(\sqrt{\alpha_{v,w}}-1\right)
		\right).
	\end{align*}
	With $1/2\le t\le 2$, $1/2\le v<2$, $1/2\le w<(2+3v)/4$ as stated in Theorem \ref{thm:main}, we have $\alpha_{v,w}\ge 1$. Thus all the (reduced) restrictions are nonnegative over $t\in [1/2,2]$ and by Lemma~\ref{lem:strong-bound-property-B} we conclude that $g_{v,w}^1$ is nonnegative over $\Gset$. The nonnegativity of $g_{v,w}^2$ follows immediately by symmetry.
	
	It remains to prove extremality. Let
	$
	q=q_1+q_2$ where $q_1,q_2\in\Ncone(\Gset).
	$
	We will show that, if $q$ is any ray listed in Theorem~\ref{thm:main}, then $q_1$ and $q_2$ are both nonnegative multiples of $q$. 
	
	Suppose first that $q=\ell^2$ is an affine square ray with two distinct points \(\vx,\vy\in\Gset\) satisfying 
	\(\ell(\vx)=\ell(\vy)=0\).  Observe that we then have $q_i(\vx)=q_i(\vy)=0$ for $i\in\{1,2\}$. Let $\vd$ be any feasible direction at $\vx$, then we have $\frac{d}{dt}q_i(\vx + t\vd)\!\mid_{t=0}\ge 0$ for all $i$, while $\sum_{i\in\{1,2\}}\frac{d}{dt}q_i(\vx + t\vd)\!\mid_{t=0} = \frac{d}{dt}\ell^2(\vx + t\vd)\!\mid_{t=0} = 2\ell(\vx)\langle \nabla \ell(\vx), \vd\rangle = 0$. Thus the directional derivatives $\frac{d}{dt}q_i(\vx + t\vd)\!\mid_{t=0} = \langle \nabla q_i(\vx), \vd\rangle= 0$ for all $i$. For any $\vx\in \Gset$, there always exist two linearly independent feasible directions, so we have $\nabla q_i(\vx) = 0$ for all $i$. Similarly we can also show that $\nabla q_i(\vy) = 0$ for all $i$. Consequently, $2R_i(\vy - \vx) = \nabla q_i(\vy) - \nabla q_i(\vx) = 0 $ and thus $R_i$ has nonzero kernel. Since \(q_i(\vx)=0\) and \(\nabla q_i(\vx)=0\), we have \(q_i(\vz)=(\vz-\vx)^\top R_i(\vz-\vx)\). Since $\vy-\vx$ is in the kernel of $R_i$, the matrix $R_i$ has rank at most one, and its nonzero eigenvalue must be nonnegative since \(q_i\geq0\) on \(\Gset\). Therefore, \(q_i\) is a nonnegative multiple of the square of an affine function whose zero line passes through \(\vx\) and \(\vy\), and hence \(q_i\) is a nonnegative multiple of \(\ell^2\).
	
	For the hyperbola ray case with $q=\ell_\Gamma$, note that both $q_i$'s
	vanish on the hyperbolic arc \(\Gamma\). A quadratic that vanishes on $\Gamma$ is a scalar multiple of \(\ell_\Gamma\), and nonnegativity guarantee that $q_i$'s are both nonnegative multiples of $\ell_\Gamma$.
	
	For the convex hull boundary-product case with $q=\ell_j\ell_k$ where $j$ and $k$ are distinct elements in $\{AB, AC, BC\}$, note that $q$ (and hence $q_i$'s) vanishes on the line segment that $j$ and $k$ represents. But the only ray that vanishes on both line segments $j$ and $k$ is a nonnegative multiple of $\ell_j\ell_k$, so $q_i$'s are both spanned by $q$. 
	
For the bounded lifted tangent rays, direct calculation yields
\begin{equation}\label{eq:f-curve}
	P_{f_u}(t)
	=
	(t-\tfrac12)(2-t)(t-u)^2.
\end{equation}
Consider \(f_u=q_1+q_2\), where
\(q_1,q_2\in\Ncone(\Gset)\). We claim that for all $i\in\{1,2\}$, $P_{q_i}= \theta_i P_{f_u}$ for some $\theta_i\ge 0$. With the claim, we know that $P_{ q_i- \theta_i f_u}$ is a zero polynomial, and by observing its coefficients in \eqref{eq:Pq} we have $q_i- \theta_i f_u = \delta_i \ell_\Gamma$ for some $\delta_i$. Observing that $f_u$ vanishes at $A$, by nonnegativity we know that $q_i$ both vanishes at $A$, and hence $\delta_i \ell_\Gamma$ vanishes at $A$. But $\ell_\Gamma$ is nonzero at $A$, hence $\delta_i=0$ and thus $q_i$'s are both nonnegative multiples of $f_u$. It suffices to prove the claim. 
Suppose first that \(1/2<u<2\).  Since \(P_{f_u}\) vanishes at
\(1/2\), \(u\), and \(2\) in which $u$ is a minimizer in the interior, the nonnegativity of \(P_{q_1}\) and
\(P_{q_2}\) implies that
\(P_{q_i}(1/2)=P_{q_i}(u)=P_{q_i}(2)=0\) for all $i\in\{1,2\}$.  Therefore,
\((t-1/2)(2-t)(t-u)^2\) divides \(P_{q_i}\). Noting that $P_{q_i}$'s are both  polynomials with degree at most four, we conclude that both of them are nonnegative multiples of $P_{f_u}$. 
When \(u=2\), since $P_{f_u}$ vanishes at $1/2$ and $2$, 
nonnegativity of $P_{q_i}$'s and the optimal condition at $2$ implies that \(P_{q_i}'(2)\leq0\). But by \eqref{eq:f-curve} we have $P_{f_2}'(2)=0$, so $P_{q_i}'(2)=0$. With this result, using second-order optimality condition at $2$ we also have \(P_{q_i}''(2)\geq0\). But by \eqref{eq:f-curve} we also have $P_{f_2}''(2)=0$, so $P_{q_i}''(2)=0$. With $P_{q_i}(2) = P_{q_i}'(2) = P_{q_i}''(2)=0$ we have that 
\((t-1/2)(2-t)^3\) divides \(P_{q_i}\), thus both of them are nonnegative multiples of $P_{f_2}$. Similar analysis can be applied to the case $u=1/2$ and hence the claim is proven.

	Finally, let \(g^1_{v,w}=q_1+q_2\), where
	\(w<(2+3v)/4\).  Recall from the first relation of \eqref{eq:bounded:g-bottom} that $
	b_{g^1_{v,w}\mid AC}(t)=(t-w)^2.$
	When \(w>1/2\), $t=w$ is a double zero of $b_{g^1_{v,w}\mid AC}(t)$ in $[1/2,2]$. Nonnegativity of the restrictions of $q_i$'s requires that $t=w$ is also a zero there. Any zero on the quadratic boundary restriction is always a double zero. Consequently, 
	\[
	b_{q_i\mid AC}(t)=\lambda_i(t-w)^2,
	\qquad
	\lambda_i\geq0,\ \forall i\in \{1,2\}.
	\]
	This conclusion also holds when \(w=1/2\). Indeed, since $t=1/2$ is a zero of $b_{g^1_{v,w}\mid AC}(t)$, nonnegativity of the restrictions of $q_i$'s requires that $t=1/2$ is also a zero of $b_{q_i\mid AC}(t)$ and hence a minimizer in interval $[1/2,2]$.  Thus we have \(b_{q_i\mid AC}'(1/2)\geq0\) and
	\[
	b_{q_1\mid AC}'(\tfrac12)+b_{q_2\mid AC}'(\tfrac12)=b'_{g^1_{v,w}\mid AC}(\tfrac12) =0.
	\]
	Hence both derivatives vanish.  Since each \(b_{q_i\mid AC}\) has degree at most
	two, it follows that
	\[
	b_{q_i\mid AC}(t)=\lambda_i(t-\tfrac12)^2,
	\qquad
	\lambda_i\geq0.
	\]
	The difference \(q_i-\lambda_i g^1_{v,w}\) vanishes on \(AC\),
	at \(B\), and to second order at \((v,v^{-1})\).  If
	\(v>1/2\), Lemma~\ref{lem:contact-uniqueness}\textup{(i)}
	makes the difference zero.  If \(v=1/2\), the curved restriction of
	\(g^1_{v,w}\) has multiplicity three there because
	\(w<(2+3v)/4\) implies
	\(\alpha_{v,w}>1\).  The same derivative
	argument as for \(f_{1/2}\) yields multiplicity at least three for every summand.
	Lemma~\ref{lem:contact-uniqueness}\textup{(ii)} again makes the
	difference zero.  Thus \(q_i=\lambda_i g^1_{v,w}\).  Coordinate exchange
	treats \(g^2_{v,w}\), completing the proof.
\end{proof}

\subsection{The PSD case}

In this subsection, we study extreme rays $q_{R,\vr,p}$ with $R\succeq 0$. Theorem \ref{thm:main} states all such extreme rays are affine squares $\ell^2$ where $\ell$ is a nonzero affine function whose zero set contains two
distinct points of \(\Gset\). In the following lemma, we show that an affine squares $\ell^2$ is extremal if and only if the such condition holds. 

\begin{lemma}\label{lem:square-criterion}
Let \(\ell\) be a nonzero affine function.  Then \(\ell^2\) spans an
extreme ray of \(\Ncone(\Gset)\) if and only if the zero set of $\ell$ contains two
distinct points of \(\Gset\).
\end{lemma}

\begin{proof}
	The extremality of affine squares \(\ell^2\) whose zero set of $\ell$ contains two
	distinct points of \(\Gset\) has been proven in Lemma \ref{prop:candidates}.		
	Suppose that $\ell$ is a nonzero affine function whose zero set contains at most one point of
	\(\Gset\). To finish the proof it suffices to show that $\ell^2$ is not an extreme ray.
	
	With at most one zero in $\Gset$, the zero set of $\ell$ cannot pass through the interior of \(\Gset\). 
	Up to exchange $\ell$ and $-\ell$, we have 
	\(\ell\geq0\) on \(\Gset\).  By
	\eqref{eq:affine-cone}, we have
	\[
	\ell=a\ell_{AB}+b\ell_{AC}+c\ell_{BC} \text{ for some }
	a,b,c\geq0.
	\]
	At least two coefficients are positive; otherwise $\ell$ is a nonnegative multiple of $\ell_{AB}$, $\ell_{AC}$, or $\ell_{BC}$ and the zero set contains at least two distinct points in $\Gset$. We then have a nontrivial conic decomposition
	\[
	\begin{aligned}
		\ell^2={}&a^2\ell_{AB}^2+b^2\ell_{AC}^2+c^2\ell_{BC}^2
		+2ab\ell_{AB}\ell_{AC}+2ac\ell_{AB}\ell_{BC}+2bc\ell_{AC}\ell_{BC}.
	\end{aligned}
	\]
	Hence
	\(\ell^2\) is not an extreme ray.
\end{proof}

We are now ready to finish the classification of extreme rays in the PSD case. 

\begin{proposition}\label{prop:psd}
Every extreme ray \(q_{R,\vr,p}\in\Ncone(\Gset)\) with \(R\succeq0\) is
an extreme affine-square ray listed Theorem~\ref{thm:main}\textup{(i)}.
\end{proposition}
\begin{proof}
	Consider any extreme ray \(q:=q_{R,\vr,p}\in\Ncone(\Gset)\) with \(R\succeq0\).
	By Proposition~\ref{prop:has-zero}, the quadratic \(q\) has at least one
	zero.
	We will proceed the proof in the following steps.  We first show that if \(q\) has a zero at which its gradient vanishes,
	then \(q\) spans an affine-square ray.  In particular, this conclusion
	holds whenever \(q\) has a zero in the interior of \(\Gset\).  We may
	therefore assume that every zero of \(q\) has nonzero gradient.  Under
	this assumption, we next show that extremality precludes zeros in
	\(AB^\circ\), \(AC^\circ\), at \(A\), or on $\Gamma$, finishing the proof.
	
	Suppose first that \(q(\vz)=0\) and \(\nabla q(\vz)=0\) for some
	\(\vz\in\Gset\).  By Taylor expansion we must have
	\(
	q(\vx)=(\vx-\vz)^\top R(\vx-\vz).
	\)
	Rank-one decomposition of the \(R\) yields $q(\vx) = \ell^2(\vx) + {\widetilde\ell}^2(\vx)$ where $\ell^2$ and ${\widetilde\ell}^2$ are affine square rays. Extremality forces $q$ to be exactly one of the affine square rays.
	
	From now on we may
	assume that every zero of \(q\) has nonzero gradient. In particular, $q$ does not have any zero in the interior of $\Gset$. Next, suppose that \(q\) has a zero \(\vz\) in \(AB^\circ\). In such case, by KKT condition there exists nonnegative $\lambda$ such that 
	\(\nabla q(\vz)=\lambda\nabla\ell_{AB}\). Indeed, recalling that $q$ has nonzero gradient at $\vz$, we have $\lambda>0$. Moreover, note that $q - \lambda \ell_{AB}\in \Ncone(\Gset)$ since it is convex and has a stationary point $\vz$ with zero function value. 
	Now we have conic decomposition $q = (q - \lambda\ell_{AB}) + \lambda\ell_{AB}$, in which $\ell_{AB}$ is not an extreme ray as stated in \eqref{eq:affine-decomposition}. Such decomposition contradicts with the assumption that $q$ is an extreme ray. The argument applies  by symmetry to the case when a
	zero of $q$ is on $AC^\circ$. The argument is also applicable to the case when a zero of $q$ is the vertex $A$ (by subtracting a conic combination of $\ell_{AB}$ and $\ell_{AC}$ from $q$ instead of only subtracting $\lambda \ell_{AB}$), in which case by KKT condition the gradient of $q$ at $A$ is a nonnegative combination of \(\nabla\ell_{AB}\) and
	\(\nabla\ell_{AC}\).
	
	It remains to consider the case when all zeros of \(q\) lie on
	\(\Gamma\).  The polynomial \(P_q\) cannot vanish identically.  Indeed,
	coefficient comparison in \eqref{eq:Pq} would then show that \(q\) is
	proportional to \(\ell_\Gamma\), whose nonzero multiples do not have
	positive-semidefinite quadratic parts.  Thus \(q\) has only finitely many
	zeros on \(\Gamma\). Note that any zero on $\Gamma$ is a minimizer of $q$ on $\Gset$, and hence by KKT condition has a Lagrange multiplier $\lambda$ associated to the constraint $-\ell_\Gamma\le 0$. We will show that for any zero on $\Gamma$, such Lagrange multiplier is strictly positive. For the case when \(\vz\in\Gamma^\circ\), 
	\(\ell_\Gamma\geq0\) is the only active constraint at \(\vz\) and the KKT
	condition yields
	\(
		\nabla q(\vz)
		=
		\lambda_\Gamma\nabla\ell_\Gamma(\vz),\) in which the gradient of \(q\) is nonzero at \(\vz\), so
	\(\lambda_\Gamma>0\).
	For the case when \(B\) is a zero of \(q\), the KKT condition yields $
		\nabla q(B)
		=
		\lambda_{AB}\nabla\ell_{AB}(B)
		+\lambda_\Gamma\nabla\ell_\Gamma(B)$ with multipliers $
		\lambda_{AB},\lambda_\Gamma\geq0.
	$
	If \(\lambda_\Gamma=0\), then \(\lambda_{AB}>0\) because
	\(\nabla q(B)\neq0\).  Performing Taylor expansion of $q$ at \(B\) we have a conic decomposition $
		q(\vx)
		=
		\lambda_{AB}\ell_{AB}(\vx)
		+(\vx-B)^\top R(\vx-B)
	$ in which $\lambda_{AB}>0$. Extremality would force \(q\) 
	to be proportional to \(\ell_{AB}\), which does not span an extreme ray
	due to \eqref{eq:affine-decomposition}, a contradiction. Therefore
	\(\lambda_\Gamma>0\).  A similar argument holds at \(C\).  Consequently, the KKT multiplier associated with
	\(-\ell_\Gamma\le 0\) is positive at every zero of \(q\).

	We will now use those positive multipliers to compare \(\ell_\Gamma\) with
	\(q\) near its zeros.  Fix any zero
	\(\vz=(u,u^{-1})\in\Gamma^\circ\cup\{B\}\), where $u\in [1/2,2)$.  The KKT relation yields
	\(
		\partial_2q(\vz)
		=
		\lambda_\Gamma\partial_2\ell_\Gamma(\vz)
		=
		-u\lambda_\Gamma
		<0.
	\)
	Note that the above also holds at \(B\) since \(\partial_2 \ell_{AB}(\vz) = 0\). By continuity, there exist a
	neighborhood of \(\vz\) in $\Gset$ and \(c_{\vz}>0\) such that
	\(\partial_2q\leq-c_{\vz}\) throughout the neighborhood.  For any nearby
	\((x_1,x_2)\in\Gset\), we have \(x_2\leq x_1^{-1}\) and 
	\[
	\begin{aligned}
		q(x_1,x_2)
		&=
		q(x_1,x_1^{-1})
		-\int_{x_2}^{x_1^{-1}}\partial_2q(x_1,t)\,dt\ge c_{\vz}\int_{x_2}^{x_1^{-1}}\,dt
		=
		\frac{c_{\vz}}{x_1}\ell_\Gamma(x_1,x_2) \ge \frac{c_{\vz}}{2} \ell_\Gamma(x_1,x_2),
	\end{aligned}
	\]
	where we use \(q(x_1,x_1^{-1})\geq0\) and \(\partial_2q\leq-c_{\vz}\) in the first inequality.  Thus \(\ell_\Gamma(\vx)\leq2q(\vx)/c_{\vz}\) near
	\(\vz\).
	At the vertex \(C\), the KKT relation
	yields
	\(
		\partial_1q(C)
		=
		\lambda_\Gamma\partial_1\ell_\Gamma(C)
		=
		-\lambda_\Gamma/2
		<0.
	\)
	By a similar argument, we can show that for some \(c_C>0\), every nearby point
	\((x_1,x_2)\in\Gset\) in a small neighborhood of $C$ satisfies $\partial_1q\le -c_C$ and hence
	\[
	\begin{aligned}
		q(x_1,x_2)
		&=
		q(x_2^{-1},x_2)
		-\int_{x_1}^{x_2^{-1}}\partial_1q(t,x_2)\,dt
		\geq
		c_C(x_2^{-1}-x_1)
		=
		\frac{c_C}{x_2}\ell_\Gamma(x_1,x_2) \ge \frac{c_C}{2}\ell_\Gamma(x_1,x_2)
	\end{aligned}
	\]
	Since \(x_2\leq2\), this proves
	\(\ell_\Gamma(\vx)\leq2q(\vx)/c_C\) near the vertex \(C\).

	Outside the above small neighborhoods near the finite number of zeros, on the
	remaining compact subset of \(\Gset\), the function $\ell_\Gamma/q$ has a maximum. Overall, there exists a constant $c>0$ such that \(
		\ell_\Gamma(\vx)\leq cq(\vx)
	\) for all $\vx\in\Gset$. Choosing \(0<\varepsilon\leq1/c\) we then have conic combination
	\(
		q
		=
		\varepsilon\ell_\Gamma
		+\bigl(q-\varepsilon\ell_\Gamma\bigr),
	\) where $\ell_\Gamma, q-\varepsilon\ell_\Gamma\in\Ncone(\Gset)$. Extremality would force
	\(q\) to be proportional to \(\ell_\Gamma\), contradicting
	\(R\succeq0\).  This completes the proof.
\end{proof}

\subsection{The non-PSD case: boundary contacts and multiplicity-dominance principle}
\label{sec:tools}

We move to the non-PSD case and start classifying all extreme rays
$q_{R,\vr,p}$ with $R\not\succeq0$.  As stated in the road map in Subsection~\ref{sec:roadmap},
all zeros of $q_{R,\vr,p}$ have to lie on the boundary $\partial\Gset$: an
interior zero would be a local minimizer, and the second-order necessary
condition would give $R\succeq0$.  From here onward, we call a point
\(\vz\in\partial\Gset\) a \emph{contact} of a quadratic \(q\) with
\(\Gset\) if \(q(\vz)=0\).  Proposition~\ref{prop:has-zero} establishes
that every non-PSD extreme ray has at least one zero.  Thus, in the non-PSD case,
the zeros of an extreme ray are precisely its contacts.

The purpose of this subsection is to convert this geometric contact
information into a sufficient certificate for non-extremality.  By
Lemma~\ref{lem:boundary-reduction-bounded}, global nonnegativity for a
quadratic $q_{R,\vr,p}$ with $R\not\succeq 0$ is determined by the three
univariate boundary restrictions in \eqref{eq:Pq}.  In this
subsection, we establish a stronger conclusion: the zeros and
multiplicities of these restrictions form a boundary-contact profile of
$q$.
After isolating the infinite-contact case, we encode this profile by
defining contact multiplicity vectors and using them to formulate the multiplicity-dominance principle in Lemma~\ref{lem:subtraction}.  This principle turns contact profile into a conic
decomposition and drives the complete non-PSD classification in the next
subsection.

We first treat the case of infinitely many contacts.  Separating this
case is necessary because the multiplicity-dominance principle relies on the multiplicities of 
nonzero boundary restrictions; an infinite contact set instead forces at
least one of those restrictions to vanish identically.

\begin{proposition}\label{prop:infinite-zeros}
Let \(q=q_{R,\vr,p}\) span an extreme ray of \(\Ncone(\Gset)\), and
suppose that \(R\not\succeq0\).  If \(q\) has infinitely many contacts,
then it is either a convex hull boundary-product ray or the hyperbola ray listed in
Theorem~\ref{thm:main}.
\end{proposition}

\begin{proof}
The boundary $\partial \Gset=AB\cup AC\cup\Gamma$ consists of three pieces, so at least one of them contains infinitely many
contacts.  If $AB$ contains infinitely many contacts, then \(b_{q\mid AB}(t) = q(1/2, t)\) has infinitely many zeros and hence vanishes
identically for all $t$. By direct computation we have $q(\vx) = q(\vx) - q(1/2,x_2) = \ell_{AB}(\vx)\ell(\vx)$ where $\ell(\vx)$ is an affine function. At any interior point $\vx$, we have $q(\vx)>0$, hence $\ell(\vx)>0$. By continuity we have $\ell(\vx)\ge 0$ for all $\vx\in\Gset$. 
Consequently, \eqref{eq:affine-cone} yields $\ell\in \cone\{\ell_{AB},\ell_{AC},\ell_{BC}\}$ and hence
\[
	q\in
	\cone\Set{
		\ell_{AB}^2,
		\ell_{AB}\ell_{AC},
		\ell_{AB}\ell_{BC}
	}.
\]
Extremality and the assumption
\(R\not\succeq0\) force $q$ to be spanned by one of the two convex hull boundary-product rays above. By symmetry we have a similar result for the case when \(AC\) contains infinitely many contacts.  Finally, if $\Gamma$ contains infinitely many contacts, i.e., if \(P_q\) has infinitely many zeros, then $P_q$ vanishes for any $t$ and hence is the
zero polynomial.  Noting that all the coefficients of $P_q$ in \eqref{eq:Pq} are $0$, we can observe that $q(\vx) = p(1-x_1x_2)$. Nonnegativity of $q$ guarantees that $q$ is spanned by the hyperbola ray.
\end{proof}

With the help of Proposition~\ref{prop:infinite-zeros}, it suffices to study only extreme rays with finite number of contacts in the non-PSD case. Let us fix one such extreme ray.  We
call
\begin{equation}\label{eq:zero-set}
	Z(q):=\Set{\vz\in\partial\Gset\mid q(\vz)=0}
\end{equation}
the \emph{contact set} of \(q\).  Its cardinality, denoted by \(\card(Z(q))\), is exactly
the number of contacts of \(q\).  Since \(Z(q)\) is finite, none of the
three boundary restrictions in \eqref{eq:Pq} vanishes
identically. For each contact \(\vz=(z_1,z_2)\in Z(q)\), at least one of the following happens: \(z_2\) is a root of \(b_{q\mid AB}\) when \(\vz\in AB\),
\(z_1\) is a root of \(b_{q\mid AC}\) when \(\vz\in AC\), or \(z_1\) is a root
of \(P_q\) when \(\vz\in\Gamma\). Recalling that any root of a univariate function has a multiplicity associated to it, we denote by \(\boldsymbol{\nu}_q(\vz)=([\vnu_q(\vz)]_{AB}, [\vnu_q(\vz)]_{AC}, [\vnu_q(\vz)]_{\Gamma})^\top\) the \emph{contact
	multiplicity vector} of \(q\) at \(\vz\), in which $[\vnu_q(\vz)]_{AB}$, $[\vnu_q(\vz)]_{AC}$, and $[\vnu_q(\vz)]_{\Gamma}$ are the multiplicities of $z_2$, $z_1$, and $z_1$ as roots of \(b_{q\mid AB}\), \(b_{q\mid AC}\), and  \(P_q\), respectively. 
	We set the multiplicity to $0$ if the associated value is not a root or if $\vz$ is not on the boundary piece. 
	
	As an example of contact set and contact multiplicity vector, consider the bounded tangent extreme ray 
	\[
	q(\vx):=f_1(\vx) = (x_1-\tfrac12)(2-x_1)
	+(x_2-\tfrac12)(2-x_2)
	-4(x_1-\tfrac12)(x_2-\tfrac12).
	\]
	By direct computation we know that $q$ is nonconvex and 
	\[
	\begin{aligned}
	b_{q\mid AB}(t)=(t-\tfrac12)(2-t),\
	b_{q\mid AC}(t)=(t-\tfrac12)(2-t),\
	P_q(t)=(t-\tfrac12)(2-t)(t-1)^2.
	\end{aligned}
	\]
	Therefore,
	\(
	Z(q)=\{A,B,C,G_1\}
	\) has cardinality $4$ and includes the three vertices $A$, $B$, $C$, and a point $G_1=(1,1)^\top$ on $\Gamma$. The contact multiplicity vector at contact $A$ is $\vnu_q(A) = (1,1,0)^\top$, in which the last component is $0$ since $A\not\in\Gamma$. Note that $A$ is on both $AB$ and $AC$ and hence contributes to two nonzero components in the contact multiplicity vector. 
	
It should be noted that the above definition of contact multiplicity vector can be applied to any quadratic \(h\) and any point \(\vz\in\partial\Gset\). A component is zero if
either \(\vz\) does not lie on the corresponding boundary piece or if the
corresponding restriction does not vanish at the associated value fo $\vz$, and is $\infty$ if $h$ vanishes at the associated boundary piece. For example, with $h(\vx) = -x_1^2$ and $\vz = (1,1)^\top$ we have $\vnu_h(\vz) = (0,0,0)^\top$, and with $\ell_\Gamma(\vx) = 1 - x_1x_2$ we have $\vnu_{\ell_\Gamma}(B) = (1, 0, \infty)^\top$.

In the following multiplicity-dominance principle lemma, we show that the relationship between contact multiplicity vectors contains important information that can be used to certify non-extremality. 
\begin{lemma}[Multiplicity-dominance principle]\label{lem:subtraction}
Suppose that \(q:=q_{R,\vr,p}\in\Ncone(\Gset)\) satisfies $R\not \succeq 0$ and $\card(Z(q))<\infty$, and that $h\in\Ncone(\Gset)$ is any nonzero ray. If
\begin{align}
	\label{eq:multiplicity-dominance}
	\boldsymbol{\nu}_h(\vz)
	\geq\boldsymbol{\nu}_q(\vz)
	\qquad
	\text{componentwise for every }\vz\in Z(q),
\end{align}
then \(q-\eps h\in\Ncone(\Gset)\) for all sufficiently small
\(\eps>0\).  Consequently, if \(q\) spans an extreme ray, then \(h\) is
spanned by \(q\).
\end{lemma}

\begin{proof}
No boundary restriction of \(q\) is identically zero because its contact
set is finite.  Consider one boundary piece $\partial\in\{AB, AC, \Gamma\}$ that contains a contact of $q$, denote the corresponding boundary 
restrictions by \(q_{\partial}\) and \(h_{\partial}\), and fix a root
\(t_0\) of \(q_{\partial}\). If \(h_{\partial}\) is not identically
zero and \(t_0\) has multiplicity \(m\) for \(q_{\partial}\), then by \eqref{eq:multiplicity-dominance} we have
\begin{align}
	q_{\partial}(t)=(t-t_0)^m a(t)
	\text{ and }
	h_{\partial}(t)=(t-t_0)^n b(t),
	\text{ where }
	n\geq m.
\end{align}
Here $a(t)$ and $b(t)$ are polynomials and \(a(t_0)\neq0\).  Hence
\(h_{\partial}/q_{\partial}=(t-t_0)^{n-m}b/a\) is bounded near
\(t_0\). Note that the above discussion implicitly assumes that \(h_{\partial}\) is not a zero polynomial, but even if it is, the boundedness of quotient \(h_{\partial}/q_{\partial}\) still holds. 
The contacts are the only points when the quotient $h_{\partial}/q_{\partial}$ needs special treatment; away from the finitely many contact, the
quotient is bounded due to compactness of $\partial \Gset$. Applying this argument on the finite number of contacts in all three
boundary pieces yields a common constant \(M>0\) such that
\[
	0\leq h(\vx)\leq Mq(\vx),
	\qquad
	\vx\in\partial\Gset.
\]
For \(0<\eps<M^{-1}\), the quadratic \(q-\eps h\) is nonnegative on the
boundary. We may also choose $\eps$ small enough so that the Hessian of $q - \eps h$ is not PSD. Thus
Lemma~\ref{lem:boundary-reduction-bounded} promotes boundary
nonnegativity to \(q-\eps h\in\Ncone(\Gset)\).  Finally,
\[
	q=(q-\eps h)+\eps h
\]
is a decomposition into elements of \(\Ncone(\Gset)\).  If the first
summand is zero, then \(h\) is already proportional to \(q\).  Otherwise,
extremality places both nonzero summands on the ray spanned by \(q\), and
again \(h\) is proportional to \(q\).
\end{proof}

Here and throughout the paper, any inequality between any two contact multiplicity vectors is always a componentwise inequality. The above multiplicity-dominance principle lemma gives a practical certificate of non-extremality:
it is enough to construct a nonnegative \(h\), not proportional to
\(q\), such that
\(\boldsymbol{\nu}_h(\vz)\geq\boldsymbol{\nu}_q(\vz)\) for every
\(\vz\in Z(q)\). 
We will use this lemma extensively in the following subsection. To demonstrate its use, we apply it to rule out the extremality of all rays with exactly one contact in the non-PSD case. 

\begin{proposition}\label{prop:one-contact-non-psd}
	Suppose that $q:=q_{R,\vr,p}\in\Ncone(\Gset)$ satisfies $R\not\succeq 0$ and $\card(Z(q))=1$. Then $q$ is not an extreme ray. 
\end{proposition}

\begin{proof}
	Let the unique contact of \(q\) be \(\vz\).  We will choose $h$ based on the location of $\vz$:
	\begin{itemize}
		\item If $\vz=A$, we set $h(\vx):=(x_1+x_2-1)^2$, so $\vnu_h(A) = (2,2,0)^\top$;
		\item If $\vz = B$, we set $h(\vx):=(x_1+\tfrac14x_2-1)^2$, so $\vnu_h(B) = (2,0,4)^\top$;
		\item If $\vz = C$, we set $h(\vx):=(x_1+4x_2-4)^2$, so $\vnu_h(C) = (0,2,4)^\top$;
		\item If $\vz\in AB^\circ$, we set $h(\vx):=\ell_{AB}^2$, so $\vnu_h(\vz) = (\infty,0,0)^\top$;
		\item If $\vz\in AC^\circ$, we set $h(\vx):=\ell_{AC}^2$, so $\vnu_h(\vz) = (0,\infty,0)^\top$;
		\item If $\vz = (u,1/u)\in \Gamma^\circ$ where $u\in(1/2, 2)$, we set  $h(\vx):=1-x_1x_2$, so $\vnu_h(\vz) = (0,0,\infty)^\top$.
	\end{itemize}
	
	Noting that the degrees of boundary restrictions $b_{q\mid AB}$, $b_{q\mid AC}$, and $P_q$ are at most $2$, $2$, and $4$ respectively, for each of the above cases we have
	\[
	\vnu_h(\vz)\geq\vnu_q(\vz).
	\]
	If $q$ is an extreme ray, then by Lemma~\ref{lem:subtraction} $h$ is spanned by $q$. However, for the listed $h$, either they are not extremal, or they have more than one contacts and hence is not spanned by $q$. Hence $q$ could not be extremal.
\end{proof}

We have now reduced the non-PSD problem to finite algebraic data.  The
contact set specifies the locations, the contact multiplicity vectors
specify the required boundary factors, and
Lemma~\ref{lem:subtraction} converts the componentwise domination
\(\boldsymbol{\nu}_h\geq\boldsymbol{\nu}_q\) into a conic
decomposition.  The next subsection follows this sequence: it enumerates the
patterns allowed by contact multiplicity vectors, constructs a multiplicity-dominance candidate
whenever possible, and performs coefficient calculations only for the
patterns that survive.

\subsection{The non-PSD case: complete extreme ray classification}
\label{sec:non-psd}

Proposition~\ref{prop:infinite-zeros} treats non-PSD extreme rays with
infinitely many contacts.  Proposition~\ref{prop:has-zero} excludes the
zero-contact case, and Proposition~\ref{prop:one-contact-non-psd} excludes the
one-contact case.  It remains to classify finite contact sets of cardinality
at least two.  In this subsection, we classify all remaining finite contact
sets.  For convenience, we write the three families of interior boundary
points as
\begin{equation}\label{eq:zero-points}
	X_t=(\tfrac12,t)\in AB^\circ,
	\qquad
	Y_t=(t,\tfrac12)\in AC^\circ,
	\qquad
	G_t=(t,t^{-1})\in\Gamma^\circ,
	\qquad
	t\in(\tfrac12,2).
\end{equation}
We will use the above notations to describe the contact set $Z(q)$ of a ray. For example, $Z(q)= \{A, X_s, Y_w, G_u\}$ means that $q$ has one contact at the vertex $A$ and three contacts on $AB^\circ$, $AC^\circ$, and $\Gamma^\circ$ respectively. When we use such notations to describe $Z(q)$, our convention is that the contacts described in the set are assumed distinct. For example, if we have $Z(q) = \{A, X_s, X_t, G_u, G_v\}$, then our convention is that $s\not =t$ and $u\not =v$.

We start with exactly two contacts. Let \(q=q_{R,\vr,p}\) be an extreme
ray with \(R\not\succeq0\) and
\(Z(q)=\{\vx,\vy\}\).  We compare \(q\) with the affine square
\(h=\ell^2\), where \(\ell=0\) is the line through \(\vx\) and
\(\vy\).  If
\[
\vnu_h(\vx)\geq\vnu_q(\vx)
\ \text{and}\ 
\vnu_h(\vy)\geq\vnu_q(\vy),
\]
then Lemma~\ref{lem:subtraction} forces \(h\) to be proportional to
\(q\), contrary to \(R\not\succeq0\).  As Lemma~\ref{lem:two-contact-profiles} will show, this comparison can fail only specific contact patterns. The following technical lemma will be used in the analysis, which is useful in eliminating some invalid contact patterns.

\begin{lemma}\label{lem:mixed-flat}
Let \(q:=q_{R,\vr,p}\in\Ncone(\Gset)\) be a ray with finite contact set.  Suppose that
\(X_s\in Z(q)\) and \(P_q(t)=\theta(t-u)^2(t-v)^2\), where
\(\theta>0\) and \(u,v\in(1/2,2)\) ($u=v$ is possible).  Then
\begin{equation}\label{eq:flat-forces-flat}
	s=\frac{u+v-1/2}{uv}\text{ and }
	b_{q\mid AC}(t)
	=\theta\left(t-u-v+\frac{uv}{2}\right)^2.
\end{equation}
In particular, \(q\) has a contact in \(AC^\circ\).
\end{lemma}

\begin{proof}
Comparing $P_q$ in the assumption with the coefficients described in \eqref{eq:Pq} yields
$
	R_{11}=\theta,$ $
	r_1=-\theta(u+v),$ $
	R_{22}=\theta u^2v^2,$ $
	r_2=-\theta uv(u+v),$, and $
	2R_{12}+p=\theta(u^2+4uv+v^2)$.
Since \(X_s\) is associated with an interior zero of the nonnegative quadratic
\(b_{q\mid AB}\), observing the description of \(b_{q\mid AB}\) in \eqref{eq:Pq} we also have
\(b_{q\mid AB}(t)=R_{22}(t-s)^2 = \theta u^2v^2(t-s)^2\).  Comparing its remaining
coefficients and eliminating \(R_{12}\) and \(p\) yields
\[
	(uvs-u-v+1/2)(uvs-4uv+u+v-1/2)=0.
\]
There are two possible roots of $s$ to the above equation; by observing the coefficient of $s$ we know that the two root sum to \(4\).  By the fact that \(2uv-u-v+1/2=(2u-1)(2v-1)/2>0\), one of the possible root
\((u+v-1/2)/(uv)\) is less than \(2\). Therefore the other possible root exceeds \(2\) and should be excluded since $s\in (1/2, 2)$.
Substitution yields \eqref{eq:flat-forces-flat}. 
\end{proof}

With the help of the above lemma, we are able to characterize the contact pattern when there are exactly two contacts, as stated below. 

\begin{lemma}
	\label{lem:two-contact-reduction}
	\label{lem:two-contact-profiles}
	Suppose that \(q:=q_{R,\vr,p}\in\Ncone(\Gset)\) spans an extreme ray,
	\(R\not\succeq0\), and \(\card(Z(q))=2\).  Then, up to exchange of the
	coordinates, either \(Z(q)=\{A,B\}\), or
	\(Z(q)=\{B,Y_w\}\) for some \(w\in(1/2,2)\).  In either case,
	\([\vnu_q(B)]_\Gamma=3\).
\end{lemma}

\begin{proof}
Write \(Z(q)=\{\vx,\vy\}\), let
\(\ell=0\) be the line through both points, and set \(h:=\ell^2\).
Comparing the contact multiplicity vectors $\vnu_q$ and $\vnu_h$ we have the following results:
\begin{itemize}
  	\item If \(\vz=A\), then
  	$
  		\vnu_q(A)\leq(2,2,0)^\top\leq\vnu_h(A).
  	$
  	\item If \(\vz=B\), then either
  	\(
  		\vnu_q(B)\leq(2,0,2)^\top\leq\vnu_h(B),
  	\)
  	or $[\vnu_q(B)]_\Gamma>[\vnu_h(B)]_\Gamma$. The latter case could only happen when $[\vnu_h(B)]_\Gamma=2$. 
  	\item If \(\vz=C\), then either
  	\(
  	\vnu_q(C)\leq(0,2,2)^\top\leq\vnu_h(C),
  	\)
  	or $[\vnu_q(C)]_\Gamma>[\vnu_h(C)]_\Gamma$. The latter case could only happen when $[\vnu_h(C)]_\Gamma=2$. 
  	\item If \(\vz\in AB^\circ\), then
  	$
  		\vnu_q(\vz)=(2,0,0)^\top\leq\vnu_h(\vz).
  	$
  	\item If \(\vz\in AC^\circ\), then
  	$
  		\vnu_q(\vz)=(0,2,0)^\top\leq\vnu_h(\vz).
  	$
  	\item If \(\vz\in\Gamma^\circ\), then either
  	\(
  		\vnu_q(\vz)=(0,0,2)^\top\leq\vnu_h(\vz),
  	\)
  	or
  	\(
  		\vnu_q(\vz)=(0,0,4)^\top\text{ and }
  		\vnu_h(\vz)=(0,0,2)^\top.
  	\)
\end{itemize}
If \(\vnu_h\geq\vnu_q\) at both contacts $\vx$ and $\vy$, the multiplicity-dominance principle in 
Lemma~\ref{lem:subtraction} and extremality force \(q\) to be proportional to the affine square \(h\), contrary to
\(R\not\succeq0\).  
As listed above, multiplicity dominance could only fail at a contact
\(G_u\in\Gamma^\circ\) (where \([\vnu_q(G_u)]_\Gamma=4\) and
\([\vnu_h(G_u)]_\Gamma=2\)), at a vertex $B$ (where 
\([\vnu_q(B)]_\Gamma\in\{3,4\}\) and
\([\vnu_h(B)]_\Gamma=2\)), or at the vertex $C$. The case at vertex $C$ is similar to that at $B$ up to exchange of coordinates. 

Suppose first that multiplicity dominance fails at a contact
\(G_u\in\Gamma^\circ\) with \([\vnu_q(G_u)]_\Gamma=4\) and
\([\vnu_h(G_u)]_\Gamma=2\). In this case, the multiplicity requires that 
\(P_q(t)=\theta(t-u)^4\) for some \(\theta>0\). Consequently, there could only be one contact on \(\Gamma\). If the second contact lies on $AB^\circ$, then
Lemma~\ref{lem:mixed-flat} (with \(u=v\)), yields a third contact on
$AC^\circ$, contradiction our assumption that $\card(Z(q))=2$. By symmetry the second contact could not lie on $AC^\circ$ either. If the second contact is \(A\), direct coefficient
comparison using \(P_q=\theta(t-u)^4\) and \(q(A)=0\) yields
\[
	b_{q\mid AC}(t)
	=
	\theta\left(t-\frac{1}{2}\right)
	\left(t-t_*\right),\text{ where }t_*:=2-\frac{(2-u)^4}{6}. 
\]
Since
\(0<(2-u)^4<(3/2)^4<9\), we have \(t_*\in(1/2,2)\), and hence $b_{q\mid AC}(t)$ is negative on \((1/2,t_*)\), contradicting nonnegativity. 

It remains to consider the case when multiplicity dominance fails at either $B$ or $C$. By symmetry, assume that the
failure happens at \(B\). 
If \([\vnu_q(B)]_\Gamma=4\), then
\(P_q(t)=\lambda(t-1/2)^4\) for some \(\lambda>0\).  Observing that $(t-1/2)^4 = P_{\ell^2}(t)$ where $\ell(\vx):=x_1+(1/4)x_2-1$, we have that $P_{q - \lambda\ell^2}$ is a zero polynomial. Comparing with the coefficients in \eqref{eq:Pq} we have  
\(
	q - \lambda\ell^2 
	=
	\eta\ell_\Gamma
\) for some $\eta$. 
Observing that
\(
\ell(Y_{7/8})=0\) and \( \ell_\Gamma(Y_{7/8})=9/{16}\), we obtain
\(
0\le q(Y_{7/8})=(9/16)\eta,
\)
so \(\eta\geq0\). Thus $q = \lambda \ell^2 + \eta\ell_\Gamma$ is a conic decomposition, in which $\lambda>0$. Extremality then forces $q$ to be a nonnegative multiple of $\ell^2$, contradicting to the assumption that $R\not\succeq 0$. Thus we are only left with
\([\vnu_q(B)]_\Gamma=3\).

Since the degree of \(b_{q\mid AB}\) and \(P_q\) are $2$ and $4$ respectively while \([\vnu_q(B)]_{AB}\ge 1\) and \([\vnu_q(B)]_\Gamma=3\), the second contact cannot lie on either  \(AB^\circ\), or \(\Gamma^\circ\).  Thus the second contact is \(A\), \(C\),
or a point \(Y_w\in AC^\circ\).  The choice \(C\) is also impossible.
Indeed, in such case we have 
\(P_q(t)=\lambda(t-1/2)^3(2-t)=\lambda P_{f_{1/2}}(t)\) for some
\(\lambda>0\). Since $P_{q - \lambda f_{1/2}}$ is a zero polynomial, comparing the coefficients in \eqref{eq:Pq} yields 
\(q=\lambda f_{1/2}+\eta\ell_\Gamma\) for some $\eta$.  Since \(Z(q)=\{B,C\}\), we
have \(q(A)>0\), while \(f_{1/2}(A)=0\) and
\(\ell_\Gamma(A)>0\).  Hence \(\eta>0\) and $q$ not extremal, leading to contradiction. 
The remaining possibilities of the second contact are \(A\) and \(Y_w\in AC^\circ\). 
\end{proof}

For the contact patterns described in the above lemma, we can associated them to extreme ray classifications below. 

\begin{lemma}\label{lem:two-contact-B}
Suppose that \(q:=q_{R,\vr,p}\in\Ncone(\Gset)\) is an extreme ray that satisfies 
\(R\not\succeq0\). If \(Z(q)=\{A,B\}\) or \(Z(q)=\{B,Y_w\}\), then $q$ is spanned by $g^1_{1/2,1/2}$ or $g^1_{1/2,w}$ respectively, where $1/2< w<7/8$. 
\end{lemma}

\begin{proof}
Suppose first that \(Z(q)=\{B,Y_w\}\). Since $Y_w$ is associated with an interior zero of \(b_{q\mid AC}\), we have 
\(b_{q\mid AC}(t)=\lambda(t-w)^2\) for some \(\lambda>0\).  By
\eqref{eq:bounded:g-bottom}, the difference
\(q-\lambda g^1_{1/2,w}\) vanishes identically on \(AC\).  
Moreover, note from Lemma~\ref{lem:two-contact-reduction} that
\([\vnu_q(B)]_\Gamma=3\), and observe that
\[
P_{g^1_{1/2,w}}(t)
=
(t-1/2)^3
\left(t+\frac32\alpha_{1/2,w}-2\right),
\text{ where }
\alpha_{1/2,w}=\frac{64(2-w)^2}{81}.
\]
Therefore, $P_{q-\lambda g^1_{1/2,w}}$ has a zero of multiplicity at least three at \(B\). Applying 
Lemma~\ref{lem:contact-uniqueness}\textup{(ii)} we have
\(q=\lambda g^1_{1/2,w}\). The condition of $w$ in the definition of $g_{v,w}^1$ yields $1/2<w<7/8$. 

Next, suppose that \(Z(q)=\{A, B\}\). Since the vertex $A$ is a contact, looking at the boundary restriction on $AC$ we should have
\(
b_{q\mid AC}(t)=(t-1/2)\ell(t),
\)
where \(\ell\) is an affine function.  Using its value at $1/2$ and $2$ for description, we have
\[
\ell(t)
=\lambda (t-\tfrac12)
+\mu(2-t),\text{ where }\lambda:=\tfrac23\ell (2)\text{ and }\mu:=\tfrac23\ell(\tfrac12).
\]
Since \(t-1/2>0\) on
\((1/2,2]\) and \(b_{q\mid AC}\geq0\) on \([1/2,2]\), we
have \(\ell\geq0\) on this interval, including at \(t=1/2\) by
continuity.  Thus $\lambda,\mu\ge0$ and we have
\[
b_{q\mid AC}(t)
=(t-\tfrac12)
\left[\lambda(t-\tfrac12)+\mu(2-t)\right]. 
\]
The two terms in the bracket are the \(AC\)-boundary restrictions of
\(\lambda g^1_{1/2,1/2}\) and \(\mu f_{1/2}\).  Observe that
\(d:=q-\lambda g^1_{1/2,1/2}-\mu f_{1/2}\) vanishes on \(AC\), and
\(P_d\) has multiplicity at least three at \(B\).
Lemma~\ref{lem:contact-uniqueness}\textup{(ii)} yields
\(q=\lambda g^1_{1/2,1/2}+\mu f_{1/2}\).  Extremality forces $q$ to be spanned by either $g^1_{1/2,1/2}$ or $f_{1/2}$, while \(f_{1/2}\) also vanishes at \(C\).  Therefore $q$ could only be spanned by $g^1_{1/2,1/2}$. 
\end{proof}

\begin{corollary}\label{cor:two-zero-exceptions}
Suppose that \(q:=q_{R,\vr,p}\in\Ncone(\Gset)\) is an extreme ray with
\(R\not\succeq0\), and \(\card(Z(q))=2\).  Then \(q\) is a nonnegative
multiple of \(g^1_{1/2,w}\) or \(g^2_{1/2,w}\) for some
\(w\in[1/2,7/8)\).
\end{corollary}

\begin{proof}
This result is immediate from Lemmas~\ref{lem:two-contact-profiles} and 
\ref{lem:two-contact-B} (up to exchange of coordinates). 
\end{proof}

We next enumerate the contact sets of cardinality at least three before
performing any further coefficient calculations.

\begin{lemma}\label{lem:finite-possibilities}
Suppose that \(q:=q_{R,\vr,p}\in\Ncone(\Gset)\) spans an extreme ray and satisfies $R\not\succeq 0$ and \(3\le \card(Z(q))<\infty\).  Up to exchange of the coordinates, exactly one of
the following cases occurs:
\begin{enumerate}
	\item \(Z(q)=\{A,B,C\}\).
	\item \(Z(q)=\{A,B,G_v\}\).
	\item \(Z(q)=\{B,Y_w,G_v\}\).
	\item \(Z(q)=\{A,B,C,G_u\}\).
\end{enumerate}
In particular, we have \(\card(Z(q))\leq4\).
\end{lemma}

\begin{proof}
	Fix any extreme ray \(q:=q_{R,\vr,p}\in\Ncone(\Gset)\) that satisfies $R\not\succeq 0$ and \(3\le \card(Z(q))<\infty\). Focusing on its contacts,  observe that the total multiplicities at all contacts are bounded by
	\begin{align}
		\label{eq:budget}
		\sum_{\vz\in Z(q)}\vnu_q(\vz)
		\leq(2,2,4)^\top.
	\end{align}
	This is since the boundary restriction functions $b_{q\mid AB}$, $b_{q\mid AC}$, and $P_q$ are at most $2$, $2$, and $4$ respectively. 
	An interior contact of $q$ on either
	$AB$ or $AC$ has multiplicity two in the corresponding component, and thus by \eqref{eq:budget}
	the quadratic $q$ cannot also vanish at an
	endpoint of that side.  An interior contact in \(\vz\in\Gamma^\circ\)
	satisfies \([\vnu_q(\vz)]_{\Gamma}\geq2\), so by \eqref{eq:budget} we have at most two contacts in \(\Gamma^\circ\). We are now ready to enumerate all possibilities of contact patterns. 
	
	If both $AB^\circ$ and $AC^\circ$ contain an interior zero, no vertex is available and
	the only possible patterns of $Z(q)$ are either
	three contacts \(\{X_s,Y_w,G_v\}\) or four contacts \(\{X_s,Y_w,G_u,G_v\}\).  If only \(AC^\circ\)
	contains an interior zero, then vertices \(A,C\) are excluded from $Z(q)$. Thus the only possible patterns of $Z(q)$ are either
	\(\{B,Y_w,G_v\}\) or \(\{Y_w,G_u,G_v\}\); the latter becomes
	\(\{X_s,G_u,G_v\}\) after exchanging the coordinates.  The contact patterns in
	which only \(AB^\circ\) has an interior zero can be enumerated by symmetry.
	
	If neither $AB^\circ$ nor $AC^\circ$ has an interior zero, all contacts are among
	\(A,B,C\) and at most two points of \(\Gamma^\circ\). 
	If in addition no contact is on \(\Gamma^\circ\), then the only contact pattern we have is $\{A, B, C\}$ since $\card(Z(q))\ge 3$. 
	By \eqref{eq:budget}, two contacts on \(\Gamma^\circ\) will force \(B,C\) to be excluded from $Z(q)$, leaving only contact pattern
	\(\{A,G_u,G_v\}\).  One contact on $\Gamma^\circ$ can occur with two or three vertices included in $Z(q)$, yielding 
	\(\{A,B,G_v\}\), \(\{B,C,G_u\}\), or \(\{A,B,C,G_u\}\), up to coordinate
	exchange. Combining the cases above, we obtain the following nine patterns:
\[
\begin{gathered}
	\{A,B,C\},\quad
	\{A,B,G_v\},\quad
	\{B,C,G_u\},\quad
	\{B,Y_w,G_v\},\quad
	\{A,G_u,G_v\},\\
	\{X_s,G_u,G_v\},\quad
	\{X_s,Y_w,G_v\},\quad
	\{A,B,C,G_u\},\quad
	\{X_s,Y_w,G_u,G_v\}.
\end{gathered}
\]
We will finish the proof by showing that only the four listed contact patterns could survive. 

First suppose that \(X_s,Y_w,G_v\in Z(q)\).  In such case, the two quadratic boundary 
restrictions $b_{q\mid AB}$ and $b_{q\mid AC}$ both have interior contacts. Such interior contacts always have multiplicity two, so we have
\(
	b_{q\mid AB}(t)=\alpha(t-s)^2\) and \(
	b_{q\mid AC}(t)=\beta(t-w)^2\) for some $\alpha$ and $\beta$. Noting that both restrictions take the value \(q(A)>0\) at \(t=1/2\), and hence 
\(
	\alpha(s-1/2)^2=\beta(w-1/2)^2>0.
	\) In particular, $\alpha,\beta>0$.
Defining 
\(
	\ell_{s,w}(x_1,x_2)
	:=(s-1/2)(x_1-w)+(w-1/2)(x_2-1/2)
\) and 
\(
	\lambda
	:=\alpha/{(w-1/2)^2}>0,
\)
a direct coefficient computation yields 
\[
	q=\lambda \ell_{s,w}^2+\eta\ell_{AB}\ell_{AC}, \text{ where }
	\eta:=2R_{12}-2\lambda(s-\tfrac12)(w-\tfrac12)\text{ satisfies }\partial_1q(X_s)=\eta(s-\tfrac12).
\]
Here $\eta\ge 0$: the optimality condition at $X_s$ yields $\partial_1q(X_s) = \langle \nabla q(X_s), (1,0)^\top\rangle \ge 0$ since $(1,0)^\top$ is a feasible direction at $X_s$. Thus the above equality on $q$ is a conic combination. Extremality forces that $q$ is a nonnegative multiple of $\ell_{s,w}^2$, contradicting the assumption that $R\not\succeq 0$. Hence contact patterns \(\{X_s,Y_w,G_v\}\) and
\(\{X_s,Y_w,G_u,G_v\}\) should both be eliminated.

Next suppose that \(Z(q)=\{B,C,G_u\}\).  Since all contacts are on $\Gamma$ and $G_u$ is on $\Gamma^\circ$, the boundary restriction $P_q$ must satisfy
\(
	P_q(t)=\kappa(t-\tfrac12)(2-t)(t-u)^2
	       =\kappa P_{f_u}(t)
\) for some $\kappa>0$. Therefore $P_{q-\kappa f_u}$ is a zero quartic polynomial. Recalling the coefficients of $P_q$ in \eqref{eq:Pq} we have $q-\kappa f_u = \delta \ell_\Gamma$ for some $\delta$. 
Noting that \(A\notin Z(q)\), \(f_u(A)=0\) and
\(\ell_\Gamma(A)>0\), evaluating at \(A\) yields \(\delta>0\). Thus $q$ has a nontrivial conic decomposition $q = \kappa f_u + \delta \ell_\Gamma$, contradicting its extremality. Hence contact pattern \(\{B,C,G_u\}\) should be eliminated. 

Now suppose that \(Z(q)=\{A,G_u,G_v\}\).  The two
interior double roots of $P_q$ yields
\(
	P_q(t)=\kappa(t-u)^2(t-v)^2
\)
for some $\kappa>0$. 
Comparing coefficients with $P_q$ in \eqref{eq:Pq} yields 
$
	R_{11}=\kappa$, $
	r_1=-\kappa(u+v)$, $
	R_{22}=\kappa u^2v^2$, $
	r_2=-\kappa uv(u+v)$, and $
	2R_{12}+p=\kappa(u^2+4uv+v^2)$.
Together with \(q(A)=0\), these identities give
\[
	b_{q\mid AC}(t)
	=\kappa(t-\tfrac12)
	 \left(t-t_*\right)\text{ where }t_*:=2-\frac{(2-u)^2(2-v)^2}{6}.
\]
Because \(u,v\in(1/2,2)\), we have 
$
	0<(2-u)^2(2-v)^2<(3/2)^4<9,
$
so \(t_*\in(\tfrac12,2)\).  The restriction $b_{q\mid AC}(t)$ is now negative on
\((\tfrac12,t_*)\), contradicting \(q\geq0\) on \(AC\). Thus we should eliminate contact pattern \(\{A,G_u,G_v\}\).

Finally, if \(Z(q)=\{X_s,G_u,G_v\}\), then
Lemma~\ref{lem:mixed-flat} supplies an additional contact
\(Y_w\in AC^\circ\). Thus \(\{X_s,G_u,G_v\}\) should not be a valid contact pattern. Indeed, even with an additional $Y_w$, the contact pattern \(\{X_s,Y_w,G_u,G_v\}\) has already been eliminated previously.

The only remaining patterns are now the four listed in the statement. In particular, we can observe that \(\card(Z(q))\leq4\).
\end{proof}

We are now ready to enumerate all the possible contact patterns listed in Lemma~\ref{lem:finite-possibilities} and decide the possible extreme rays which such patterns. We have the following lemmas.

\begin{lemma}\label{lem:pattern_ABC}
	Suppose that \(q:=q_{R,\vr,p}\in\Ncone(\Gset)\) is an extreme ray that satisfies $R\not\succeq 0$ and \(Z(q)=\{A,B,C\}\). Then $q$ is spanned by either
	\(f_{1/2}\) or \(f_2\).
\end{lemma}

\begin{proof}
	Since
	\(Z(q)=\{A,B,C\}\), the boundary restrictions of $q$ at $AB$ and $AC$ have to be
\(
	b_{q\mid AB}(t)=\mu(t-1/2)(2-t)\) and $
	b_{q\mid AC}(t)=\lambda(t-1/2)(2-t)$ for some $
	\lambda,\mu>0$. We will study the contact multiplicity vector components  \(m_B:=[\vnu_q(B)]_\Gamma\) and
\(m_C:=[\vnu_q(C)]_\Gamma\).  
The finiteness of \(Z(q)\) ensures that
\(P_q\) is nonzero.  Since \(P_q\) has degree at most four, we have
\(m_B,m_C\geq1\) and \(m_B+m_C\leq4\).

We first rule out the possibility that \(m_B,m_C\geq2\). In such case we have
\(m_B=m_C=2\), and hence the boundary restriction of $q$ on $\Gamma$ has to be 
\(
	P_q(t)=\kappa(t-1/2)^2(2-t)^2\) for some \(
	\kappa>0.
\) Comparing the coefficients of \(P_q\) and \(b_{q\mid AB}\) in \eqref{eq:Pq} yields
\(R_{22}=\kappa>0\) and \(R_{22}=-\mu<0\), a contradiction.  Therefore,
we should have either \(m_B=1\) or \(m_C=1\).

Suppose that \(m_C=1\). Then we have \(m_B\leq3\). 
Direct computation then yield
\[
\begin{aligned}
	\vnu_{f_{1/2}}(A)
	&=(1,1,0)^\top=\vnu_q(A),\\
	\vnu_{f_{1/2}}(B)
	&=(1,0,3)^\top\geq(1,0,m_B)^\top=\vnu_q(B),\\
	\vnu_{f_{1/2}}(C)
	&=(0,1,1)^\top=\vnu_q(C).
\end{aligned}
\]
Lemma~\ref{lem:subtraction} and extremality forces \(q\) to be spanned by 
\(f_{1/2}\). In the case when \(m_B=1\), the same argument with symmetry forces $q$ to be spanned by \(f_2\). Hence the claim is proven. 
\end{proof}

\begin{lemma}\label{lem:pattern_ABG}
	Suppose that \(q:=q_{R,\vr,p}\in\Ncone(\Gset)\) is an extreme ray that satisfies $R\not\succeq 0$ and \(Z(q)=\{A,B,G_v\}\). Then $q$ is spanned by
	\(g^1_{v,1/2}\).
\end{lemma}

\begin{proof}
	The proof here is quite close to that in Lemma~\ref{lem:two-contact-B}. 
Since the vertice $A$ is a contact, looking at the boundary restriction on $AC$ we should have
\(
	b_{q\mid AC}(t)=(t-1/2)\ell(t),
\)
where \(\ell\) is an affine function.  Using its value at $1/2$ and $2$ for description, we have
\[
\ell(t)
=\lambda (t-\tfrac12)
+\mu(2-t),\text{ where }\lambda:=\tfrac23\ell (2)\text{ and }\mu:=\tfrac23\ell(\tfrac12).
\]
Since \(t-1/2>0\) on
\((1/2,2]\) and \(b_{q\mid AC}\geq0\) on \([1/2,2]\), we
have \(\ell\geq0\) on this interval, including at \(t=1/2\) by
continuity.  Thus $\lambda,\mu\ge0$ and we have
\[
	b_{q\mid AC}(t)
	=(t-\tfrac12)
	\left[\lambda(t-\tfrac12)+\mu(2-t)\right].
\]
Observe that the two terms in brackets are the \(AC\)-boundary restrictions of
\(g^1_{v,1/2}\) and \(f_v\), respectively.  Setting
\[
	h:=q-\lambda g^1_{v,1/2}-\mu f_v,
\]
Then \(h\) vanishes identically on \(AC\).  All three quadratics in the definition of $h$ vanish at \(B\), so \(h(B)=0\).  All the three quadratics' boundary restriction on $\Gamma^\circ$ 
also have a double zero at \(G_v\), and hence $P_h$ has a zero of multiplicity at least two there.
Lemma~\ref{lem:contact-uniqueness}\textup{(i)} yields \(h=0\). 
 Therefore
\(
	q=\lambda g^1_{v,1/2}+\mu f_v.
\)
Extremality of $q$ forces it to be spanned by either $g^1_{v,1/2}$ or $f_v$. But the contact set of $f_v$ consists of an additional contact $C$, so $q$ could only be spanned by $g^1_{v,1/2}$.
\end{proof}

\begin{lemma}\label{lem:pattern_BYG}
	Suppose that \(q:=q_{R,\vr,p}\in\Ncone(\Gset)\) is an extreme ray that satisfies $R\not\succeq 0$ and \(Z(q)=\{B,Y_w,G_v\}\). Then $q$ is spanned by
	\(g^1_{v,w}\), where
	\(1/2<w<{(2+3v)}/{4}\).
\end{lemma}

\begin{proof}
Since $Y_w$ is a contact on $AC^\circ$, the boundary restriction of $q$ at $AC$ has to be 
\(
	b_{q\mid AC}(t)=\lambda(t-w)^2\) where $\lambda>0$.
Noting that \(b_{g^1_{v,w}\mid AC}(t)=(t-w)^2\), the difference 
\(
	q-\lambda g^1_{v,w}
\)
vanishes identically on \(AC\).  It also vanishes at \(B\), and its
restriction on the hyperbolic boundary piece has a double zero at \(G_v\).
Lemma~\ref{lem:contact-uniqueness}\textup{(i)} gives \(q-\lambda g^1_{v,w}=0\), so
$q$ is spanned by \(g^1_{v,w}\). 
It remains to recover the admissible range of \(w\).  Note that
\[
	P_{g^1_{v,w}}(t)
	=
	(t-\tfrac12)(t-v)^2
	\left(t+\frac32\alpha_{v,w}-2\right)\ge 0,\ \forall t\in [1/2, 2].
\]
Consequently, for nonnegativity of $P_{g^1_{v,w}}(t)$ to hold we need \(1/2+(3/2)\alpha_{v,w}-2\ge 0\), i.e., 
$	\alpha_{v,w}\geq1$. 
Since \(v,w<2\), applying the definition of $\alpha_{v,w}$ to the above inequality we have $
	4(2-w)\geq3(2-v)$, i.e., $
	w\leq{(2+3v)}/{4}.
$
Note that the equality could not be achieved in the above inequality; otherwise $q = \lambda g_{v,w}^1$ is an affine squares described in \eqref{eq:g-overlap} and contradicts the assumption that $R\not\succeq 0$. Therefore the admissible range of  $w$ is 	\(1/2<w<{(2+3v)}/{4}\).
\end{proof}

\begin{lemma}\label{lem:pattern_ABCG}
	Suppose that \(q:=q_{R,\vr,p}\in\Ncone(\Gset)\) is an extreme ray that satisfies $R\not\succeq 0$ and \(Z(q)=\{A,B,C,G_u\}\). Then $q$ is spanned by
	\(f_u\).
\end{lemma}

\begin{proof}
	Since $B, C, G_u$ are contacts on $\Gamma$ and $G_u$ is an on $\Gamma^\circ$, the boundary restriction $P_q$ must satisfy
	\(
	P_q(t)=\kappa(t-\tfrac12)(2-t)(t-u)^2
	=\kappa P_{f_u}(t)
	\) for some $\kappa>0$. Therefore $P_{q-\kappa f_u}$ is a zero quartic polynomial. Recalling the coefficients of $P_q$ in \eqref{eq:Pq} we have $q-\kappa f_u = \delta \ell_\Gamma$ for some $\delta$. 
	Noting that \(q(A) = f_u(A)=0\) and
	\(\ell_\Gamma(A)>0\), evaluating at \(A\) yields \(\delta=0\). Thus $q$ is a nonnegative multiple of $f_u$. 
\end{proof}

We are now ready to classify all extreme rays in the non-PSD case.

\begin{proposition}
	\label{prop:non-PSD-complete}
	If $q:=q_{R,\vr,p}\in\Ncone(\Gset)$ is an extreme ray with $R\not\succeq 0$, then it is in the list stated in Theorem \ref{thm:main}. 
\end{proposition}
\begin{proof}
	If $q$ has infinitely many contacts,
	Proposition~\ref{prop:infinite-zeros} places it in the list.  If $q$ has
	finitely many contacts, Propositions~\ref{prop:has-zero} and
	\ref{prop:one-contact-non-psd} show that it has at least two contacts.
	Corollary~\ref{cor:two-zero-exceptions} treats the case of exactly two
	contacts, while Lemmas~\ref{lem:finite-possibilities}--\ref{lem:pattern_ABCG}
	treat the case of at least three contacts.  Thus $q$ belongs to the list.
\end{proof}

\subsection{Proof of the main theorem}

We have finished the extreme ray classification and is now ready to prove the main theorem by showing that $\Ncone(\Gset)$ is indeed the conic hull of all the extreme rays listed in Theorem \ref{thm:main}.

\begin{proof}[Proof of Theorem~\ref{thm:main}]
Proposition~\ref{prop:candidates} proves nonnegativity and
extremality of every listed ray in Theorem \ref{thm:main}.  Proposition~\ref{prop:psd}
shows that every extreme ray $q_{R,\vr,p}$ with $R\succeq 0$ is in the list.  For an extreme ray $q_{R,\vr,p}$ with $R\not\succeq 0$,   Proposition~\ref{prop:non-PSD-complete}
 shows that it is in the list. 
Thus the list contains every extreme ray.

It remains to prove the conic-hull assertion.  The cone
\(\Ncone(\Gset)\) is closed because it is the intersection of the closed
halfspaces \(q(\vx)\geq0\), one for each \(\vx\in\Gset\).  It is pointed:
if both \(q\) and \(-q\) belong to the cone, then \(q\) vanishes in the
interior of \(\Gset\) and is the zero polynomial. 
We will now look at a convex base 
\[
	\mathcal A:=\Set{q\in\Ncone(\Gset)\mid\Lambda(q)=1}\text{ where }\Lambda(q):=\int_{\Gset}q(\vx)\,d\vx.
\]
The linear functional $\Lambda$ is positive on every nonzero member of
\(\Ncone(\Gset)\); therefore the above set $\mathcal A$ 
is a nonempty convex set containing exactly one normalized point from
each nonzero ray.  Nonemptiness follows, for example, by normalizing the
constant polynomial \(1\).  The set $\mathcal A$ is closed as the intersection of the convex cone $\Ncone(\Gset)$ and a hyperplane.  It is also bounded:
otherwise, for an unbounded sequence \(q_k\) in the set, choose a
subsequence such that \(\lVert q_k\rVert\to\infty\), where the norm is any
norm on the six-dimensional coefficient space.  A further subsequence of
\(q_k/\lVert q_k\rVert\) converges to a coefficient vector \(q_\infty\)
of norm one.  Closedness of the cone yields
\(q_\infty\in\Ncone(\Gset)\), while continuity of \(\Lambda\) yields
$
	\Lambda(q_\infty)
	=
	\lim_{k\to\infty}{\Lambda(q_k)}/{\lVert q_k\rVert}
	=0,
$
contradicting strict positivity of \(\Lambda\) on nonzero members of the
cone $\Ncone(\Gset)$.  Thus the set $\mathcal A$ is compact.
Its extreme points are precisely the extreme rays of
\(\Ncone(\Gset)\), normalized by \(\Lambda\).  Indeed, a nontrivial
decomposition of a ray in \(\Ncone(\Gset)\) becomes a nontrivial convex combination of two points in $\mathcal A$ after
normalization by $\Lambda$, and the converse follows by removing the
normalization.  By Minkowski's theorem the
set $\mathcal A$ is the convex hull of the normalized extreme rays listed in Theorem \ref{thm:main}.  Rescaling proves that \(\Ncone(\Gset)\) is their conic hull.
\end{proof}

\section{Characterization of the lifted convex hull $\cC(\Gset)$}
\label{sec:G}

Using the extreme-ray classification in Theorem~\ref{thm:main}, we now
construct a finite semidefinite representation of $\cC(\Gset)$.  The key
observation is that every affine-square ray belongs to $\Smat_+^3$, while
every other extreme ray listed in Theorem~\ref{thm:main} vanishes at either
$B=(\tfrac12,2)$ or $C=(2,\tfrac12)$.  Recall from the introduction section the definitions of
\[
    \Ncone_B = \Set{q\in\Ncone(\Gset)\mid q(B)=0}
    \text{ and }
    \Ncone_C = \Set{q\in\Ncone(\Gset)\mid q(C)=0}.
\]
Theorem~\ref{thm:main} then gives
\begin{equation}\label{eq:NG_decomposition}
    \Ncone(\Gset)=\Smat_+^3+\Ncone_B+\Ncone_C.
\end{equation}
Indeed, every extreme ray of $\Ncone(\Gset)$ belongs to one of the three
cones on the right-hand side.  Conversely, each of these cones is contained
in $\Ncone(\Gset)$.

It remains to describe $\Ncone_B$ and $\Ncone_C$.  
By symmetry, it suffices
to treat $\Ncone_B$.  With the help of Lemma \ref{lem:strong-bound-property-B}, we have an
intersection representation
\[
    \Ncone_B
    =\Ncone_1\cap\Ncone_2\cap\Ncone_3\cap\mathcal H_B,
\]
where
\begin{align}
\label{eq:N1-sdr}
    \Ncone_1
    &:=\Set{\mxRrp\ \middle|\
      \widehat b_{q\mid AB}(t)\geq0
      \quad\forall\,t\in[\tfrac12,2]},\\
\label{eq:N2-sdr}
    \Ncone_2
    &:=\Set{\mxRrp\ \middle|\
      b_{q\mid AC}(t)\geq0
      \quad\forall\,t\in[\tfrac12,2]},\\
\label{eq:N3-sdr}
    \Ncone_3
    &:=\Set{\mxRrp\ \middle|\
      \widehat P_q(t)\geq0
      \quad\forall\,t\in[\tfrac12,2]},
\end{align}
and
\begin{equation}\label{eq:HB}
    \mathcal H_B
    :=\Set{\mxRrp\ \middle|\
    \begin{pmatrix}
        1&\tfrac12&2\\
        \tfrac12&\tfrac14&1\\
        2&1&4
    \end{pmatrix}
    \mathbin{\bullet}\mxRrp=0}.
\end{equation}
Note that $\Ncone_1$, $\Ncone_2$, and $\Ncone_3$ are all descriptions of polynomials nonnegative over the interval $[1/2,2]$. We are now ready to characterize $\Ncone_B$, as stated in the following direct corollary of Lemma \ref{lem:strong-bound-property-B}.

\begin{corollary}
\label{cor:NB_as_intersection}
The cone $\Ncone_B$ satisfies
\[
    \Ncone_B
    =\Ncone_1\cap\Ncone_2\cap\Ncone_3\cap\mathcal H_B.
\]
Moreover, the three cones on the right admit the following fixed-size
semidefinite representations:
\begin{align}
\label{eq:N1-explicit}
    \Ncone_1
    ={}&\Set{\mxRrp\ \middle|\
        -R_{12}-\tfrac52R_{22}-2r_2\geq0,\quad
        -R_{12}-4R_{22}-2r_2\geq0},\\
\label{eq:N2-explicit}
    \Ncone_2
    ={}&\Set{\mxRrp\ \middle|\
        \begin{gathered}
        \exists\,\mu\geq0:\\[-2pt]
        \begin{pmatrix}
            \tfrac14R_{22}+r_2+p+\mu
            &\tfrac12(R_{12}+2r_1-\tfrac52\mu)\\
            \tfrac12(R_{12}+2r_1-\tfrac52\mu)
            &R_{11}+\mu
        \end{pmatrix}\succeq0
        \end{gathered}},\\
\label{eq:N3-explicit}
    \Ncone_3
    ={}&\Set{\mxRrp\ \middle|\
        \begin{gathered}
        \exists\,\Lambda^-,\Lambda^+\in\Smat_+^2
        \text{ satisfying}\\
        2r_2+R_{12}+\tfrac12p+\tfrac12r_1+\tfrac18R_{11}
        =-\tfrac12\lambda^-_{00}+2\lambda^+_{00},\\
        2R_{12}+p+r_1+\tfrac14R_{11}
        =\lambda^-_{00}-\lambda^-_{01}
         +4\lambda^+_{01}-\lambda^+_{00},\\
        2r_1+\tfrac12R_{11}
        =2\lambda^-_{01}-\tfrac12\lambda^-_{11}
         +2\lambda^+_{11}-2\lambda^+_{01},\\
        R_{11}=\lambda^-_{11}-\lambda^+_{11},
        \end{gathered}}
\end{align}
where
\[
    \Lambda^-
    =\begin{pmatrix}
        \lambda^-_{00}&\lambda^-_{01}\\
        \lambda^-_{01}&\lambda^-_{11}
    \end{pmatrix},
    \qquad
    \Lambda^+
    =\begin{pmatrix}
        \lambda^+_{00}&\lambda^+_{01}\\
        \lambda^+_{01}&\lambda^+_{11}
    \end{pmatrix}.
\]
\end{corollary}

\begin{proof}
The intersection formula follows directly from Lemma~\ref{lem:strong-bound-property-B}.
The univariate representations below are standard forms of the
Markov--Luk\'acs theorem; see, for example,
\cite{blekherman2012semidefinite}.  For $\Ncone_1$, the polynomial
$\widehat b_{q\mid AB}$ is linear, so its nonnegativity on
$[\tfrac12,2]$ is equivalent to nonnegativity at the two endpoints, which
gives \eqref{eq:N1-explicit}.

For $\Ncone_2$, the Markov--Luk\'acs theorem gives
\[
    b_{q\mid AC}(t)
    =\begin{pmatrix}1&t\end{pmatrix}
      \Lambda
      \begin{pmatrix}1\\t\end{pmatrix}
      +\mu(t-\tfrac12)(2-t),
    \qquad
    \Lambda\succeq0,
    \quad
    \mu\geq0.
\]
Matching coefficients and eliminating the entries of $\Lambda$ gives
\eqref{eq:N2-explicit}.

Finally, a polynomial of degree at most three is nonnegative on
$[\tfrac12,2]$ if and only if it can be written as
\[
    (t-\tfrac12)
    \begin{pmatrix}1&t\end{pmatrix}
    \Lambda^-
    \begin{pmatrix}1\\t\end{pmatrix}
    +(2-t)
    \begin{pmatrix}1&t\end{pmatrix}
    \Lambda^+
    \begin{pmatrix}1\\t\end{pmatrix},
    \qquad
    \Lambda^-,\Lambda^+\succeq0.
\]
Applying this representation to $\widehat P_q$ and matching coefficients
gives \eqref{eq:N3-explicit}.
\end{proof}

We next dualize this description and characterize the lifted convex hull $\cC(\Gset)$.  Let
\[
    z_B:=\begin{pmatrix}1\\[1pt]\tfrac12\\[1pt]2\end{pmatrix},
    \qquad
    M_1(t):=
    \begin{pmatrix}
        0&0&-1\\
        0&0&-\tfrac12\\
        -1&-\tfrac12&-(t+2)
    \end{pmatrix}.
\]
Define
\begin{align}
\label{eq:D1}
    \mathcal D_1
    :={}&\Set{
        \lambda_-M_1(\tfrac12)+\lambda_+M_1(2)
        \ \middle|\
        \lambda_-,\lambda_+\geq0},\\
\label{eq:D2}
    \mathcal D_2
    :={}&\Set{
        \begin{pmatrix}
            y_0&y_1&\tfrac12y_0\\
            y_1&y_2&\tfrac12y_1\\
            \tfrac12y_0&\tfrac12y_1&\tfrac14y_0
        \end{pmatrix}
        \ \middle|\
        \begin{pmatrix}y_0&y_1\\y_1&y_2\end{pmatrix}\succeq0,
        \quad
        \tfrac52y_1-y_2-y_0\geq0},\\
\label{eq:D3}
    \mathcal D_3
    :={}&\Set{
        \begin{pmatrix}
            y_1+\tfrac12y_0
            &y_2+\tfrac12y_1+\tfrac14y_0
            &y_0\\
            y_2+\tfrac12y_1+\tfrac14y_0
            &y_3+\tfrac12y_2+\tfrac14y_1+\tfrac18y_0
            &y_1+\tfrac12y_0\\
            y_0
            &y_1+\tfrac12y_0
            &0
        \end{pmatrix}
        \ \middle|\
        \begin{gathered}
        \begin{pmatrix}
            y_1-\tfrac12y_0&y_2-\tfrac12y_1\\
            y_2-\tfrac12y_1&y_3-\tfrac12y_2
        \end{pmatrix}\succeq0,\\[4pt]
        \begin{pmatrix}
            2y_0-y_1&2y_1-y_2\\
            2y_1-y_2&2y_2-y_3
        \end{pmatrix}\succeq0
        \end{gathered}}.
\end{align}

\begin{theorem}
\label{thm:baseline_hull}
We have
\begin{equation}\label{eq:CF}
    \cC(\Gset)
    =\Set{\mxX\ \middle|\
        \mxX\succeq0,\quad
        \mxX\in\Ncone_B^*,\quad
        \begin{pmatrix}
            1&x_2&x_1\\
            x_2&X_{22}&X_{12}\\
            x_1&X_{12}&X_{11}
        \end{pmatrix}
        \in\Ncone_B^*}.
\end{equation}
Moreover,
\begin{equation}\label{eq:NB_star}
    \Ncone_B^*
    =\mathcal D_1+\mathcal D_2+\mathcal D_3
     +\operatorname{span}\Set{z_Bz_B^\top},
\end{equation}
where $\mathcal D_1$, $\mathcal D_2$, and $\mathcal D_3$ are given in
\eqref{eq:D1}--\eqref{eq:D3}.  Equivalently, $Y\in\Ncone_B^*$ if and only
if there exist $Y_i\in\mathcal D_i$, $i=1,2,3$, and $\lambda\in\R$ such
that
\[
    Y=Y_1+Y_2+Y_3+\lambda z_Bz_B^\top.
\]
In particular, \eqref{eq:CF} is a finite semidefinite extended formulation
of $\cC(\Gset)$.
\end{theorem}

\begin{proof}
By \eqref{eq:NG_decomposition},
\[
    \Ncone(\Gset)^*
    =\Smat_+^3\cap\Ncone_B^*\cap\Ncone_C^*.
\]
The cones $\Ncone_B$ and $\Ncone_C$ are interchanged by swapping $x_1$
and $x_2$, so their duals are related by the corresponding simultaneous
row-and-column permutation.  Combining this observation with
\eqref{eq:KNC_relations} and the normalization $Y_{00}=1$ gives
\eqref{eq:CF}.

It remains to establish \eqref{eq:NB_star}.  The cone $\Ncone_B$ satisfies
the relative Slater condition in $\mathcal H_B$.  Indeed,
\[
    q_0(\vx):=(2-x_2)+(1-x_1x_2)
\]
belongs to $\Ncone_B$, and its three interval polynomials are
\[
    \widehat b_{q_0\mid AB}(t)=\tfrac32,
    \qquad
    b_{q_0\mid AC}(t)=\tfrac52-\tfrac12t,
    \qquad
    \widehat P_{q_0}(t)=2t,
\]
which are strictly positive on $[\tfrac12,2]$.  Therefore
\[
    \Ncone_B^*
    =\Ncone_1^*+\Ncone_2^*+\Ncone_3^*+\mathcal H_B^\perp.
\]
The dual of the cone of nonnegative linear polynomials on
$[\tfrac12,2]$ is generated by evaluation at the two endpoints; applying
the adjoint of $q\mapsto\widehat b_{q\mid AB}$ gives
$\Ncone_1^*=\mathcal D_1$.  The degree-two truncated moment cone on the
same interval is
\[
    \Set{(y_0,y_1,y_2)\ \middle|\
    \begin{pmatrix}y_0&y_1\\y_1&y_2\end{pmatrix}\succeq0,
    \quad
    \tfrac52y_1-y_2-y_0\geq0},
\]
and applying the adjoint of $q\mapsto b_{q\mid AC}$ gives
$\Ncone_2^*=\mathcal D_2$.  Likewise, the degree-three truncated moment
cone on $[\tfrac12,2]$ is described by the two localizing matrices in
\eqref{eq:D3}; applying the adjoint of $q\mapsto\widehat P_q$ gives
$\Ncone_3^*=\mathcal D_3$.  Finally,
\[
    \mathcal H_B^\perp
    =\operatorname{span}\Set{z_Bz_B^\top},
\]
which proves \eqref{eq:NB_star}.
\end{proof}

\section{Conclusion} \label{sec:conclusion}

We studied the quadratic convexification of the hyperbolically truncated square $\Gset$. Although \(\conv(\Gset)\) is a triangle, its lifted convex hull retains substantially richer information about the curved boundary. We completely characterized the extreme rays of the cone \(\Ncone(\Gset)\) of quadratic functions nonnegative on \(\Gset\). In addition to affine-square, boundary-product, and hyperbola rays, the classification contains parameterized families of bounded lifted tangent and bitangent rays. The proof combines a separation of the PSD and non-PSD cases with an analysis of the locations and multiplicities of boundary contacts.

Using this classification, we obtained an exact finite semidefinite representation of \(\cC(\Gset)\). The key step is to organize the non-square extreme rays through the two exposed faces associated with the endpoints of the hyperbolic arc and to reduce membership in these faces to univariate polynomial nonnegativity on a fixed interval. While this work follows the general extreme-ray framework of our earlier analysis of an unbounded product-constrained region, compact truncation produces different contact patterns and a distinct finite conic formulation. In particular, organizing the non-square extreme rays through the two exposed faces associated with the endpoints of the hyperbolic arc converts the continuously parameterized tangent and bitangent families into an exact finite semidefinite representation of the lifted convex hull.

\section{Acknowledgment}
The first and second authors are partially supported by AFOSR grant FA9550-25-1-0278. This work was supported in part by OpenAI API credits provided by Clemson University and administered by Clemson University Research Computing and Data (RCD). It also used open-weight language models hosted by Clemson University and made available through the Clemson RCD LLM Service.

Artificial intelligence tools were used during the development of this manuscript after the authors had established the main proof strategy. Specifically, the authors first identified the special structure of the non-PSD case, which reduces nonnegativity to the boundary and allows the analysis to focus on boundary-contact patterns.
They also established that extremality imposes rigid conditions on the locations and multiplicities of the contacts, reducing the possible contact configurations to a small number.
Based on the discovery, the authors then built a road map for completing the classification of all extreme rays of $\Ncone(\Gset)$.   
The OpenAI ChatGPT 5.6 (Sol) model was then used to assist in generating candidate extreme rays and preliminary arguments for analyzing these remaining configurations and constructing the lifted convex hull semidefinite representation. 
The authors independently reviewed and verified all AI-assisted output, corrected it where necessary, and rewrote the proofs to provide more natural, rigorous, and interpretable arguments. The authors take full responsibility for the correctness of all results, proofs, and conclusions presented in the paper.

\bibliography{article}


\begin{thebibliography}{6}
\ifx \bisbn   \undefined \def \bisbn  #1{ISBN #1}\fi
\ifx \binits  \undefined \def \binits#1{#1}\fi
\ifx \bauthor  \undefined \def \bauthor#1{#1}\fi
\ifx \batitle  \undefined \def \batitle#1{#1}\fi
\ifx \bjtitle  \undefined \def \bjtitle#1{#1}\fi
\ifx \bvolume  \undefined \def \bvolume#1{\textbf{#1}}\fi
\ifx \byear  \undefined \def \byear#1{#1}\fi
\ifx \bissue  \undefined \def \bissue#1{#1}\fi
\ifx \bfpage  \undefined \def \bfpage#1{#1}\fi
\ifx \blpage  \undefined \def \blpage #1{#1}\fi
\ifx \burl  \undefined \def \burl#1{\textsf{#1}}\fi
\ifx \doiurl  \undefined \def \doiurl#1{\url{https://doi.org/#1}}\fi
\ifx \betal  \undefined \def \betal{\textit{et al.}}\fi
\ifx \binstitute  \undefined \def \binstitute#1{#1}\fi
\ifx \binstitutionaled  \undefined \def \binstitutionaled#1{#1}\fi
\ifx \bctitle  \undefined \def \bctitle#1{#1}\fi
\ifx \beditor  \undefined \def \beditor#1{#1}\fi
\ifx \bpublisher  \undefined \def \bpublisher#1{#1}\fi
\ifx \bbtitle  \undefined \def \bbtitle#1{#1}\fi
\ifx \bedition  \undefined \def \bedition#1{#1}\fi
\ifx \bseriesno  \undefined \def \bseriesno#1{#1}\fi
\ifx \blocation  \undefined \def \blocation#1{#1}\fi
\ifx \bsertitle  \undefined \def \bsertitle#1{#1}\fi
\ifx \bsnm \undefined \def \bsnm#1{#1}\fi
\ifx \bsuffix \undefined \def \bsuffix#1{#1}\fi
\ifx \bparticle \undefined \def \bparticle#1{#1}\fi
\ifx \barticle \undefined \def \barticle#1{#1}\fi
\bibcommenthead
\ifx \bconfdate \undefined \def \bconfdate #1{#1}\fi
\ifx \botherref \undefined \def \botherref #1{#1}\fi
\ifx \url \undefined \def \url#1{\textsf{#1}}\fi
\ifx \bchapter \undefined \def \bchapter#1{#1}\fi
\ifx \bbook \undefined \def \bbook#1{#1}\fi
\ifx \bcomment \undefined \def \bcomment#1{#1}\fi
\ifx \oauthor \undefined \def \oauthor#1{#1}\fi
\ifx \citeauthoryear \undefined \def \citeauthoryear#1{#1}\fi
\ifx \endbibitem  \undefined \def \endbibitem {}\fi
\ifx \bconflocation  \undefined \def \bconflocation#1{#1}\fi
\ifx \arxivurl  \undefined \def \arxivurl#1{\textsf{#1}}\fi
\csname PreBibitemsHook\endcsname

\bibitem[\protect\citeauthoryear{Anstreicher and
  Burer}{2010}]{AnstreicherBurer2010}
\begin{barticle}
\bauthor{\bsnm{Anstreicher}, \binits{K.M.}},
\bauthor{\bsnm{Burer}, \binits{S.}}:
\batitle{Computable representations for convex hulls of low-dimensional
  quadratic forms}.
\bjtitle{Mathematical Programming}
\bvolume{124},
\bfpage{33}--\blpage{43}
(\byear{2010})
\doiurl{10.1007/s10107-010-0355-9}
\end{barticle}
\endbibitem

\bibitem[\protect\citeauthoryear{Sherali and
  Adams}{2013}]{sherali2013reformulation}
\begin{bbook}
\bauthor{\bsnm{Sherali}, \binits{H.D.}},
\bauthor{\bsnm{Adams}, \binits{W.P.}}:
\bbtitle{A Reformulation-linearization Technique for Solving Discrete and
  Continuous Nonconvex Problems}.
\bpublisher{Springer},
\blocation{Berlin}
(\byear{2013})
\end{bbook}
\endbibitem

\bibitem[\protect\citeauthoryear{Belotti et~al.}{2010}]{belotti2010valid}
\begin{barticle}
\bauthor{\bsnm{Belotti}, \binits{P.}},
\bauthor{\bsnm{Miller}, \binits{A.J.}},
\bauthor{\bsnm{Namazifar}, \binits{M.}}:
\batitle{Valid inequalities and convex hulls for multilinear functions}.
\bjtitle{Electronic Notes in Discrete Mathematics}
\bvolume{36},
\bfpage{805}--\blpage{812}
(\byear{2010})
\end{barticle}
\endbibitem

\bibitem[\protect\citeauthoryear{Anstreicher
  et~al.}{2021}]{AnstreicherBurerPark2021}
\begin{barticle}
\bauthor{\bsnm{Anstreicher}, \binits{K.M.}},
\bauthor{\bsnm{Burer}, \binits{S.}},
\bauthor{\bsnm{Park}, \binits{K.}}:
\batitle{Convex hull representations for bounded products of variables}.
\bjtitle{Journal of Global Optimization}
\bvolume{80},
\bfpage{757}--\blpage{778}
(\byear{2021})
\doiurl{10.1007/s10898-021-01046-7}
\end{barticle}
\endbibitem

\bibitem[\protect\citeauthoryear{Zhang et~al.}{2026}]{zhang2026nonnegative}
\begin{botherref}
\oauthor{\bsnm{Zhang}, \binits{Y.}},
\oauthor{\bsnm{Ouyang}, \binits{Y.}},
\oauthor{\bsnm{Yang}, \binits{B.}}:
Nonnegative quadratics over a quadrant with a bilinear constraint.
arXiv preprint arXiv:2608.16836
(2026)
\end{botherref}
\endbibitem

\bibitem[\protect\citeauthoryear{Blekherman
  et~al.}{2012}]{blekherman2012semidefinite}
\begin{bbook}
\bauthor{\bsnm{Blekherman}, \binits{G.}},
\bauthor{\bsnm{Parrilo}, \binits{P.A.}},
\bauthor{\bsnm{Thomas}, \binits{R.R.}}:
\bbtitle{Semidefinite Optimization and Convex Algebraic Geometry}.
\bpublisher{SIAM},
\blocation{Philadelphia, PA}
(\byear{2012})
\end{bbook}
\endbibitem

\end{thebibliography}

\end{document}